\documentclass[11pt,reqno,twoside]{amsart}
\usepackage{amssymb,amsmath,amsthm,soul,color,paralist}
\usepackage{t1enc}
\usepackage[cp1250]{inputenc}
\usepackage{a4,indentfirst,latexsym}
\usepackage{graphics}
\usepackage{mathrsfs}
\usepackage{cite,enumitem,graphicx}
\usepackage[colorlinks=true,urlcolor=blue,
citecolor=red,linkcolor=blue,linktocpage,pdfpagelabels,
bookmarksnumbered,bookmarksopen]{hyperref}
\usepackage[english]{babel}
\usepackage[left=2.50cm,right=2.50cm,top=2.72cm,bottom=2.72cm]{geometry}
\usepackage[metapost]{mfpic}
\usepackage[hyperpageref]{backref}
\usepackage[colorinlistoftodos]{todonotes}
\usepackage[normalem]{ulem}

\makeatletter
\providecommand\@dotsep{5}
\def\listtodoname{List of Todos}
\def\listoftodos{\@starttoc{tdo}\listtodoname}
\makeatother

\numberwithin{equation}{section}

\newcommand{\la}{\lambda}
\newcommand{\ri}{\rightarrow}
\newcommand{\q}{q^{*}_{s}}
\newcommand{\p}{p^{*}_{s}}

\newcommand{\R}{\mathbb{R}}
\newcommand{\X}{\mathbb{X}}

\newcommand{\C}{\mathcal{C}}
\newcommand{\B}{\mathcal{B}}
\newcommand{\J}{\mathcal{J}}
\newcommand{\I}{\mathcal{I}}
\newcommand{\N}{\mathcal{N}}

\newcommand{\A}{\mathcal{A}}

\DeclareMathOperator{\dive}{div}
\DeclareMathOperator{\supp}{supp}
\DeclareMathOperator{\e}{\varepsilon}

\newtheorem{theorem}{Theorem}[section]
\newtheorem{corollary}{Corollary}

\newtheorem{lemma}[theorem]{Lemma}
\newtheorem{proposition}{Proposition}

\theoremstyle{definition}

\keywords{Fractional $p\&q$ Laplacians; variational methods; Ljusternik-Schnirelman theory}
\subjclass[2010]{
47G20, 
35R11, 
35A15, 
58E05} 

\date{}

\begin{document}
\title[Fractional $p\&q$ Laplacian]{Existence, multiplicity and concentration for a class of fractional $p\&q$ Laplacian problems in $\R^{N}$}

\author[C. O. Alves]{Claudianor O. Alves}
\address{Claudianor O. Alves\hfill\break\indent 
Universidade Federal de Campina Grande,\hfill\break\indent
Unidade Academica de Matematica\hfill\break\indent
CEP: 58429-900, Campina Grande-PB, Brazil}
\email{coalves@mat.ufcg.edu.br}

\author[V. Ambrosio]{Vincenzo Ambrosio}
\address{Vincenzo Ambrosio\hfill\break\indent 
Dipartimento di Ingegneria Industriale e Scienze Matematiche \hfill\break\indent
Universit\`a Politecnica delle Marche\hfill\break\indent
Via Brecce Bianche, 12\hfill\break\indent
60131 Ancona (Italy)}
\email{ambrosio@dipmat.univpm.it}

\author[T. Isernia]{Teresa Isernia}
\address{Teresa Isernia\hfill\break\indent
Dipartimento di Ingegneria Industriale e Scienze Matematiche \hfill\break\indent
Universit\`a Politecnica delle Marche\hfill\break\indent
Via Brecce Bianche, 12\hfill\break\indent
60131 Ancona (Italy)}
\email{isernia@dipmat.univpm.it}

\begin{abstract}
In this work we consider the following class of fractional $p\&q$ Laplacian problems
\begin{equation*}
(-\Delta)_{p}^{s}u+ (-\Delta)_{q}^{s}u + V(\e x) (|u|^{p-2}u + |u|^{q-2}u)= f(u) \mbox{ in } \R^{N},
\end{equation*}
where $\e>0$ is a parameter, $s\in (0, 1)$, $1< p<q<\frac{N}{s}$, $(-\Delta)^{s}_{t}$, with $t\in \{p,q\}$, is the fractional $t$-Laplacian operator, $V:\R^{N}\rightarrow \R$ is a continuous potential and $f:\R\rightarrow \R$ is a $\C^{1}$-function with subcritical growth.
Applying minimax theorems and the Ljusternik-Schnirelmann theory, we investigate the existence, multiplicity and concentration of nontrivial solutions provided that $\e$ is sufficiently small.
\end{abstract}

\maketitle

\section{Introduction}

In this paper we are concerned with the existence, multiplicity and concentration results of nonnegative solutions for the following class of fractional $p\&q$ Laplacian problems
\begin{equation}\label{P}
(-\Delta)_{p}^{s}u+ (-\Delta)_{q}^{s}u + V(\e x) (|u|^{p-2}u + |u|^{q-2}u)= f(u) \quad \mbox{ in } \R^{N},
\end{equation}
where $\e>0$ is a small parameter, $s\in (0, 1)$, $1< p<q<\frac{N}{s}$, $V:\R^{N}\ri \R$ is a continuous function verifying the following condition introduced by Rabinowitz in \cite{Rab}:
\begin{equation*} \tag{$V_{0}$}\label{V0}
V_{\infty}:= \liminf_{|x|\ri \infty} V(x) >V_{0}:= \inf_{\R^{N}} V(x)>0,
\end{equation*}
and the nonlinearity $f\in \C^{1}(\R, \R)$ satisfies the following conditions:
\begin{compactenum}
\item [$(f_{1})$] $f(t)=0$ for all $t\leq 0$;
\item [$(f_{2})$] $\displaystyle{\lim_{|t|\ri 0} \frac{|f(t)|}{|t|^{p-1}}=0}$;
\item [$(f_{3})$] there exists $r\in (q, \q)$ such that $\displaystyle{\lim_{|t|\ri \infty} \frac{|f(t)|}{|t|^{r-1}}=0}$, where $\q= \frac{Nq}{N-sq}$;
\item [$(f_{4})$] $\displaystyle{\lim_{|t|\ri \infty} \frac{F(t)}{t^{q}}=\infty}$, where $\displaystyle{F(t)=\int_{0}^{t} f(\tau) d\tau}$;
\item [$(f_{5})$] $\displaystyle{\frac{f(t)}{t^{q-1}}}$ is increasing for $t>0$.
\end{compactenum}

The operator $(-\Delta)^{s}_{t}$, with $t\in \{p,q\}$, is the fractional $t$-Laplacian which may be defined for any $u\in \C^{\infty}_{c}(\R^{N})$ by
\begin{equation*}
(-\Delta)_{t}^{s}u(x):=2\, \lim_{\e\rightarrow 0} \int_{\R^{N}\setminus \B_{\e}(x)} \frac{|u(x)-u(y)|^{t-2}(u(x)-u(y))}{|x-y|^{N+st}}\, dxdy \quad (x\in \R^{N}).
\end{equation*}
When $s=1$, equation \eqref{P} becomes a $p\&q$ elliptic problem of the form
\begin{equation}\label{pq}
-\Delta_{p}u-\Delta_{q} u+|u|^{p-2}u+|u|^{q-2}u= f(x, u) \quad \mbox{ in } \R^{N}.
\end{equation}
As explained in \cite{CIL}, one of the mean reasons of studying \eqref{pq} is connected to the more general reaction-diffusion system:
$$
u_{t}=\dive(D(u)\nabla u)+c(x,u) \mbox{ and } D(u)=|\nabla u|^{p-2}+|\nabla u|^{q-2},
$$
which appears in biophysics, plasma physics and chemical reaction design. Indeed, in these applications, $u$  stands for a concentration, $\dive(D(u) \nabla u)$ is the diffusion with diffusion coefficient $D(u)$, and the reaction term $c(x, u)$ relates to source and loss processes. We point out that classical $p\&q$ Laplacian problems in bounded or unbounded domains have been studied by several authors; see for instance \cite{BF, CIL, Fig1, Fig2, HL, LG, LL, MP} and the references therein.

However, the study of the fractional $p$-Laplacian operator has achieved a tremendous popularity in the last decade.
For instance, in \cite{FrP, LLq} the authors studied fractional $p$-eigenvalue problems. Regularity results for weak solutions have been established in \cite{DKP1, DKP2, IMS, KMS}. We also mention \cite{A1, A2, AI3, FP, MMB, Torres} for
different existence and multiplicity results for problems in bounded domains or in the whole of $\R^{N}$.
More in general, nonlocal operators and fractional spaces are extensively studied due to their great application in several contexts such as obstacle problem, optimization, finance, phase transition, material science, anomalous diffusion, soft thin films, multiple scattering, quasi-geostrophic flows, water waves, and so on. For more details we refer to \cite{DPV, MBRS}.

When $p=q=2$, \eqref{P} is equivalent to the well-known fractional Schr\"odinger equation of the type
\begin{equation}\label{FSE}
\e^{2s}(-\Delta)^{s}u+ V(x) u= f(u) \quad \mbox{ in } \R^{N},
\end{equation}
which appears when we look for standing wave solutions $\psi(x,t)=e^{-\frac{\imath c t}{\e}}$of the following time dependent fractional Schr\"odinger equation
$$
\imath \e \frac{\partial \psi}{\partial t}=\e^{2s}(-\Delta)^{s}\psi+W(x)\psi-f(|\psi|) \quad \mbox{ in } \R^{N}\times \R.
$$
The above equation was derived by Laskin and plays a fundamental role in the study of fractional quantum mechanics; see \cite{Laskin1} for more details. In the last two decades many authors studied existence, multiplicity and concentration of nontrivial solutions to \eqref{FSE} assuming different conditions on the potential $V(x)$ and considering nonlinearities with subcritical or critical growth; see \cite{AM, A3, AI, AI2, DPMV, FQT, FS, Secchi}. For instance, in \cite{FS}, the authors used Nehari manifold arguments and  Ljusternik-Schnirelmann theory to deduce the existence and multiplicity of solutions to \eqref{FSE} requiring that $f$ is a $\C^{1}$-function with subcritical growth and verifying the Ambrosetti-Rabinowitz condition \cite{AR}:
\begin{equation}\tag{AR}\label{AR}
\exists \mu>2 \, : \, 0<\mu F(t)\leq tf(t) \quad \forall t>0.
\end{equation}
This assumption is quite natural when we investigate superlinear problems because it
guarantees that Palais-Smale sequence of the energy functional associated to the problem under consideration is bounded. However, \eqref{AR} is very restrictive and eliminates many nonlinearities. Therefore, several authors sought to introduce conditions weaker than \eqref{AR}; see for instance \cite{J, LW, MS, SZ, SW}.

In the present paper we deal with multiple solutions for the fractional problem \eqref{P} applying the Ljusternik-Schnirelmann category theory and without requiring \eqref{AR}. We emphasize that in the papers treating the existence of multiple solutions via Ljusternik-Schnirelmann category, a fundamental step is the verification of $(PS)$ condition (Palais-Smale condition), which is not proved in most papers without \eqref{AR}, because the Cerami's sequence works better for problem with this type of nonlinearity. For example, in the local setting, in \cite{LW} the authors considered a Schr\"odinger equation under a compactness condition on the potential $V$. Again, in \cite{SW},  the authors proved the $(PS)$ condition for functionals of the type $\Phi(u)=\frac{\|u\|^{2}}{2}-I(u)$, assuming that $I$ is weakly continuous.
In the nonlocal framework, we mention the papers \cite{AlvAmb, A4, AH} in which the existence of positive solutions for \eqref{FSE} is considered without requiring \eqref{AR}.

Motivated by the above facts and by \cite{AFpq, FS}, in this work we show that it is possible to verify the $(PS)$ condition on a Nehari manifold without requiring \eqref{AR}, and this represents the novelty in the study of problems like \eqref{P}.
We point out that as far as we know, in literature appear only few papers on fractional $p\&q$ Laplacian problems \cite{Apq, CB}, but no results on the multiplicity of solutions for problem \eqref{P} are available. So the aim of this work is to give a first result in this direction.

Since we deal with the multiplicity of solutions to \eqref{P}, we recall that if $Y$ is a given closed set of a topological space $X$, we denote by $cat_{X}(Y)$ the Ljusternik-Schnirelmann category of $Y$ in $X$, that is the least number of closed and contractible sets in $X$ which cover $Y$; see \cite{W}.

Now we state our main results.
\begin{theorem}\label{thm2}
Assume that \eqref{V0} and $(f_{1})$-$(f_{5})$ hold. Then, there exists $\bar{\e}>0$ such that problem \eqref{P} has a nonnegative ground state solution $u_{\e}$ for all $\e \in (0, \bar{\e})$. Moreover, for each sequence $\e_{n}\ri 0$, there is a subsequence such that for each $n\in \mathbb{N}$, the solution $u_{\e_{n}}$ concentrates
around a point $x_{0}\in \R^{N}$ such that $V(x_{0})= V_{0}$. More precisely, there exists $C>0$ such that for all $\delta>0$, there exist $\bar{R}>0$ and $n_{0}\in \mathbb{N}$ such that
\begin{align*}
\int_{\R^{N}\setminus \B_{\e_{n}\bar{R}}(y)} f(u_{\e_{n}})u_{\e_{n}}\,dx<\e_{n}^{N}\delta \quad \mbox{ and } \quad \int_{\B_{\e_{n}\bar{R}}(y)} f(u_{\e_{n}})u_{\e_{n}}\,dx\geq C\e_{n}^{N}
\end{align*}
for all $n\geq n_{0}$.
\end{theorem}

\begin{theorem}\label{thm1}
Under the assumptions of Theorem \ref{thm2}, for any $\delta>0$ there exists $\e_{\delta}>0$ such that, for any $\e \in (0, \e_{\delta})$, problem \eqref{P} has at least $cat_{M_{\delta}}(M)$ nonnegative and nontrivial solutions, where
\begin{equation*}
M=\{x\in \R^{N} : V(x)= V_{0}\} \quad \mbox{ and } \quad M_{\delta}= \{x\in \R^{N} : dist(x, M)\leq \delta\}.
\end{equation*}
\end{theorem}

\begin{theorem}
Under the assumptions of Theorem \ref{thm2}, every solution to \eqref{P} is bounded.
\end{theorem}
The proof of the above results is obtained applying variational methods and borrowing some ideas developed in \cite{AFpq} to study a class of quasilinear problems which includes the $p\&q$ elliptic case. Anyway, we can not repeat the same arguments exploited in \cite{AFpq} since in our context we have to take care of the appearance of fractional
$p\&q$ Laplacian operators and that our nonlinearity does not verify \eqref{AR}.
For these reasons, we first prove some technical lemmas  which allow us to overcome some difficulties coming from the nonlocal character of the involved fractional operators; see also \cite{AI3}.
Hence, we deal with the existence of solutions for the autonomous problem associated to \eqref{P}. We note that the proof of boundedness of Palais-Smale sequences is completely different from the one given in \cite{AFpq} in which \eqref{AR} is not assumed; see Lemma \ref{lem3}.
After that, we study the existence of solutions to \eqref{P} and,  taking into account some ideas present in \cite{AP, AT}, we  consider the concentration behavior of solutions. We note that the concentration phenomenon obtained in this work is in the integral sense and it is not the same considered in \cite{AFpq}. Indeed, in our framework, it seems very hard to prove that the solutions go to zero at infinity because H\"older continuous regularity results like \cite{HL} (when $s=1$) and \cite{IMS} (for $s\in (0,1)$ and $p=q\in (1, \infty)$), are not currently available for fractional $p\&q$ Laplacian operators.
Subsequently, we combine Nehari manifold arguments and Ljusternik-Schnirelmann category theory to deduce a multiplicity result for \eqref{P}. Finally, we use a variant of the Moser iteration argument \cite{Moser} to get the boundedness of solutions to \eqref{P}.

The paper is organized as follows. In Section \ref{Sect2} we collect some preliminary results. In Section \ref{Sect3} we deal with autonomous fractional $p\&q$ Laplacian problems. In Section \ref{Sect4} we give the proof of Theorem \ref{thm2}. The Section \ref{Sect5} is devoted to the multiplicity of solutions to \eqref{P}. In the last section we prove the boundedness of solutions to \eqref{P}.

\section{Preliminary}\label{Sect2}
\noindent
In this preliminary section we recall some facts about the fractional Sobolev spaces and we prove some technical lemmas which we will use later.

Let $1\leq p\leq \infty$ and $A\subset \R^{N}$. We denote by $|u|_{L^{p}(A)}$ the $L^{p}(A)$-norm of a function $u:\R^{N}\ri \R$ belonging to $L^{p}(A)$. When $A=\R^{N}$, we simply write $|u|_{p}$. We define $\mathcal{D}^{s, p}(\R^{N})$ as the closure of $\C^{\infty}_{c}(\R^{N})$ with respect to
$$
[u]_{s, p}^{p}= \iint_{\R^{2N}} \frac{|u(x)-u(y)|^{p}}{|x-y|^{N+sp}}dxdy.
$$
Let us define $W^{s, p}(\R^{N})$ as the set of functions $u\in L^{p}(\R^{N})$ such that $[u]_{s, p}<\infty$, endowed with the norm
\begin{equation*}
\|u\|_{s, p}^{p}= [u]_{s, p}^{p}+ |u|_{p}^{p}.
\end{equation*}

We recall the following embeddings of the fractional Sobolev spaces into Lebesgue spaces.
\begin{theorem}[\cite{DPV}]\label{Sembedding}
Let $s\in (0,1)$ and $N>sp$. Then there exists a constant $S_{*}>0$
such that for any $u\in \mathcal{D}^{s, p}(\R^{N})$
\begin{equation*}
|u|^{p}_{\p} \leq S_{*}^{-1} [u]^{p}_{s, p}.
\end{equation*}
Moreover, $W^{s, p}(\R^{N})$ is continuously embedded in $L^{t}(\R^{N})$ for any $t\in [p, \p]$ and compactly in $L^{t}(\B_{R}(0))$, for all $R>0$ and for any $t\in [1, \p)$.
\end{theorem}

Proceeding as in \cite{FQT, Secchi} we can prove the next compactness-Lions type result.
\begin{lemma}\label{lemVT}
Let $N>sp$ and $m\in [p, \p)$. If $\{u_{n}\}$ is a bounded sequence in $W^{s, p}(\R^{N})$ and if
\begin{equation}\label{ter4}
\lim_{n\rightarrow \infty} \sup_{y\in \R^{N}} \int_{\B_{R}(y)} |u_{n}|^{m} dx=0
\end{equation}
for some $R>0$, then $u_{n}\rightarrow 0$ in $L^{t}(\R^{N})$ for all $t\in (p, \p)$.
\end{lemma}
\begin{proof}
Let $\tau \in (m, \p)$ and $u\in W^{s, p}(\R^{N})$. Applying H\"older and Sobolev inequality we can infer
\begin{align*}
|u|_{L^{\tau}(\B_{R}(y))} &\leq |u|_{L^{m}(\B_{R}(y))}^{1-\alpha} |u|_{L^{\p}(\B_{R}(y))}^{\alpha} \\
&\leq C |u|_{L^{m}(\B_{R}(y))}^{1- \alpha} \left ( \int_{\B_{R}(y)} \int_{\B_{R}(y)} \frac{|u(x)- u(y)|^{p}}{|x-y|^{N+sp}} dxdy + \int_{\B_{R}(y)} |u|^{p} dx\right)^{\frac{\alpha}{p}},
\end{align*}
where ${\alpha= \frac{\tau-m}{\p-m}\frac{\p}{\tau}}$.
Now, covering $\R^{N}$ by balls of radius $R$ in such a way that each point of $\R^{N}$ is contained in at most $N+1$ balls, and using the fact that $\{u_{n}\}$ is bounded in $W^{s, p}(\R^{N})$, we find
\begin{align*}
|u_{n}|_{\tau}^{\tau} \leq C \sup_{y\in \R^{N}} \left(\int_{\B_{R}(y)} |u_{n}|^{m} dx \right)^{\left(\frac{m(\p-\tau)}{\p-m}\right)}\ri 0 \quad \mbox{ as } n\ri \infty
\end{align*}
in view of \eqref{ter4}. An interpolation argument gives the thesis.
\end{proof}

The lemma below provides a way to manipulate smooth truncations for the fractional $p$-Laplacian.
Let us note that this result can be seen as a generalization of the second statement of Lemma $5$ in \cite{PP} to the case of the space $W^{s, p}(\R^{N})$ with $p\neq 2$.
\begin{lemma}\label{Psi}
Let $u\in W^{s, p}(\R^{N})$ and $\phi\in \mathcal{C}^{\infty}_{c}(\R^{N})$ be such that $0\leq \phi\leq 1$, $\phi=1$ in $\B_{1}(0)$ and $\phi=0$ in $\B_{2}^{c}(0)$. Set $\phi_{r}(x)=\phi(\frac{x}{r})$.  Then
$$
\lim_{r\rightarrow \infty} [u \phi_{r}-u]_{s, p}=0 \quad \mbox{ and } \quad \lim_{r\rightarrow \infty} |u\phi_{r}-u|_{p}=0. 
$$
\end{lemma}
\begin{proof}
Taking into account that $\phi_{r}u\rightarrow u$ a.e. in $\R^{N}$ as $r\rightarrow \infty$ and $u\in L^{p}(\R^{N})$, and invoking the Dominated Convergence Theorem we have $\lim_{r\rightarrow \infty} |u\phi_{r}-u|_{p}=0$. 
Now, we prove that $\lim_{r\rightarrow \infty} [u \phi_{r}-u]_{s, p}=0$.

Let us note that
\begin{align*}
&[u\phi_{r}-u]^{p}_{s, p}\\
&\leq 2^{p-1} \left[\iint_{\R^{2N}}\!\! \frac{|u(x)|^{p}|\phi_{r}(x)\!-\!\phi_{r}(y)|^{p}}{|x\!-\!y|^{N+sp}}dx dy\!+\!\iint_{\R^{2N}} \frac{|\phi_{r}(x)\!-\!1|^{p}|u(x)\!-\!u(y)|^{p}}{|x\!-\!y|^{N+sp}}dx dy\right] \\
&=:2^{p-1}\left[A_{r}+B_{r}\right].
\end{align*}
Exploiting $|\phi_{r}(x)-1|\leq 2$, $|\phi_{r}(x)-1|\rightarrow 0$ a.e. in $\R^{N}$ and $u\in W^{s, p}(\R^{N})$, from the Dominated Convergence Theorem it follows that $B_{r}\rightarrow 0$ as  $r\rightarrow \infty$.
Next, we aim to show that
$$
A_{r}\rightarrow 0 \quad \mbox{ as } r\rightarrow \infty.
$$
Firstly, we point out that
\begin{align*}
\R^{2N}&=((\R^{N} \setminus \B_{2r}(0))\!\times\! (\R^{N} \setminus \B_{2r}(0)))\cup (\B_{2r}(0)\!\times \! \R^{N}) \cup ((\R^{N} \setminus \B_{2r}(0))\!\times\! \B_{2r}(0))\\
&=: X^{1}_{r}\cup X^{2}_{r} \cup X^{3}_{r}.
\end{align*}
Thus
\begin{align}\label{Pa1}
&\iint_{\R^{2N}} |u(x)|^{p} \frac{|\phi_{r}(x)-\phi_{r}(y)|^{p}}{|x-y|^{N+sp}} \, dx dy \nonumber\\
&\quad =\iint_{X^{1}_{r}} |u(x)|^{p} \frac{|\phi_{r}(x)-\phi_{r}(y)|^{p}}{|x-y|^{N+sp}} \, dx dy
+\iint_{X^{2}_{r}} |u(x)|^{p} \frac{|\phi_{r}(x)-\phi_{r}(y)|^{p}}{|x-y|^{N+sp}} \, dx dy \nonumber\\
&\quad \quad + \iint_{X^{3}_{r}} |u(x)|^{p} \frac{|\phi_{r}(x)-\phi_{r}(y)|^{p}}{|x-y|^{N+sp}} \, dx dy.
\end{align}
Since $\phi=0$ in $\R^{N}\setminus \B_{2}(0)$, we have
\begin{align}\label{Pa2}
\iint_{X^{1}_{r}} |u(x)|^{p} \frac{|\phi_{r}(x)-\phi_{r}(y)|^{p}}{|x-y|^{N+sp}} \, dx dy=0.
\end{align}
Using $0\leq \phi\leq 1$, $|\nabla \phi|\leq 2$ and applying the Mean Value Theorem, we can see that
\begin{align}\label{Pa3}
\iint_{X^{2}_{r}}& |u(x)|^{p} \frac{|\phi_{r}(x)-\phi_{r}(y)|^{p}}{|x-y|^{N+sp}} \, dx dy \nonumber\\
&=\int_{\B_{2r}(0)} \,dx \int_{\{y\in \R^{N}: |x-y|\leq r\}} |u(x)|^{p} \frac{|\phi_{r}(x)-\phi_{r}(y)|^{p}}{|x-y|^{N+sp}} \, dy \nonumber \\
&\quad+\int_{\B_{2r}(0)} \, dx \int_{\{y\in \R^{N}: |x-y|> r\}} |u(x)|^{p} \frac{|\phi_{r}(x)-\phi_{r}(y)|^{p}}{|x-y|^{N+sp}} \, dy  \nonumber\\
&\leq C r^{-p} |\nabla \phi|_{\infty}^{p} \int_{\B_{2r}(0)} \, dx \int_{\{y\in \R^{N}: |x-y|\leq r\}} \frac{|u(x)|^{p}}{|x-y|^{N+sp-p}} \, dy \nonumber \\
&\quad+ C \int_{\B_{2r}(0)} \, dx \int_{\{y\in \R^{N}: |x-y|> r\}} \frac{|u(x)|^{p}}{|x-y|^{N+sp}} \, dy \nonumber\\
&\leq C r^{-sp} \int_{\B_{2r}(0)} |u(x)|^{p} \, dx+C r^{-sp} \int_{\B_{2r}(0)} |u(x)|^{p} \, dx \nonumber \\
&=Cr^{-sp} \int_{\B_{2r}(0)} |u(x)|^{p} \, dx.
\end{align}
Regarding the last integral in \eqref{Pa1} we can note that
\begin{align}\label{Pa4}
&\iint_{X^{3}_{r}} |u(x)|^{p} \frac{|\phi_{r}(x)-\phi_{r}(y)|^{p}}{|x-y|^{N+sp}} \, dx dy \nonumber\\
&=\int_{\R^{N}\setminus \B_{2r}(0)} \, dx \int_{\{y\in \B_{2r}(0): |x-y|\leq r\}} |u(x)|^{p} \frac{|\phi_{r}(x)-\phi_{r}(y)|^{p}}{|x-y|^{N+sp}} \, dy \nonumber\\
&\quad+\int_{\R^{N}\setminus \B_{2r}(0)} \,dx \int_{\{y\in \B_{2r}(0): |x-y|>r\}} |u(x)|^{p} \frac{|\phi_{r}(x)-\phi_{r}(y)|^{p}}{|x-y|^{N+sp}} \, dy=: C_{r}+ D_{r}.
\end{align}
By the Mean Value Theorem, and observing that if $(x, y) \in (\R^{N}\setminus \B_{2r}(0))\times \B_{2r}(0)$ and $|x-y|\leq r$, then $|x|\leq 3r$, we get
\begin{align}\label{Pa5}
C_{r}&\leq r^{-p} |\nabla \phi|_{\infty}^{p} \int_{\B_{3r}(0)} \, dx \int_{\{y\in \B_{2r}(0): |x-y|\leq r\}} \frac{|u(x)|^{p}}{|x-y|^{N+sp-p}} \, dy \nonumber\\
&\leq C r^{-p}  \int_{\B_{3r}(0)} |u(x)|^{p} \, dx \int_{\{z\in \R^{N}: |z|\leq r\}} \frac{1}{|z|^{N+sp-p}} \, dz \nonumber\\
&=C r^{-sp} \int_{\B_{3r}(0)} |u(x)|^{p} \, dx.
\end{align}
Note that for any $K>4$ it holds
$$
X_{r}^{3}=(\R^{N}\setminus \B_{2r}(0))\times \B_{2r}(0) \subset (\B_{K r}(0) \times \B_{2r}(0)) \cup ((\R^{N}\setminus \B_{Kr}(0))\times \B_{2r}(0)).
$$
Then, we have the following estimates
\begin{align}\label{Pa6}
\int_{\B_{Kr}(0)} \, dx &\int_{\{y\in \B_{2r}(0): |x-y|> r\}}  |u(x)|^{p} \frac{|\phi_{r}(x)-\phi_{r}(y)|^{p}}{|x-y|^{N+sp}} \, dy \nonumber\\
&\leq C \int_{\B_{Kr}(0)} \, dx \int_{\{y\in \B_{2r}(0): |x-y|> r\}} \frac{|u(x)|^{p}}{|x-y|^{N+sp}} \, dy \nonumber \\
&\leq C \int_{\B_{Kr}(0)} |u(x)|^{p} \, dx \int_{\{z\in \R^{N}: |z|> r\}} \frac{1}{|z|^{N+sp}} \, dz \nonumber\\
&= C r^{-sp} \int_{\B_{Kr}(0)} |u(x)|^{p} \, dx.
\end{align}
Now, if $(x, y)\in (\R^{N}\setminus \B_{Kr}(0))\times \B_{r}(0)$, then $|x-y|\geq |x|- |y|\geq \frac{|x|}{2}+ \frac{K}{2}r -2r >\frac{|x|}{2}$, and using H\"older inequality we can see that
\begin{align}\label{Pa7}
&\int_{\R^{N}\setminus \B_{Kr}(0)} \, dx \int_{\{y\in \B_{2r}(0): |x-y|>r\}} |u(x)|^{p} \frac{|\phi_{r}(x)-\phi_{r}(y)|^{p}}{|x-y|^{N+sp}} \, dy \nonumber\\
&\quad \leq  C \int_{\R^{N}\setminus \B_{Kr}(0)} \, dx \int_{\{y\in \B_{2r}(0): |x-y|>r \}} \frac{|u(x)|^{p}}{|x-y|^{N+sp}} \, dy \nonumber\\
&\quad \leq C r^{N} \int_{\R^{N}\setminus \B_{Kr}(0)} \frac{|u(x)|^{p}}{|x|^{N+sp}} \, dx \nonumber\\
&\quad \leq C r^{N} \left(\int_{\R^{N}\setminus \B_{Kr}(0)} |u(x)|^{p^{*}_{s}} \, dx\right)^{\frac{p}{p^{*}_{s}}} \left(\int_{\R^{N}\setminus \B_{Kr}(0)} |x|^{-(N+sp)\frac{p^{*}_{s}}{p^{*}_{s}-p}} \, dx\right)^{\frac{p^{*}_{s}-p}{p^{*}_{s}}} \nonumber\\
&\quad \leq C K^{-N} \left(\int_{\R^{N}\setminus \B_{Kr}(0)} |u(x)|^{p^{*}_{s}} \, dx\right)^{\frac{p}{p^{*}_{s}}}.
\end{align}
Therefore, combining (\ref{Pa6}) and (\ref{Pa7}), we have
\begin{align}\label{Pa8}
D_{r}\leq C r^{-sp} \int_{\B_{Kr}(0)} |u(x)|^{p} \, dx+C K^{-N}.
\end{align}
Putting together (\ref{Pa1})-(\ref{Pa5}) and (\ref{Pa8}), we can infer
\begin{align*}
\iint_{\R^{2N}} |u(x)|^{p} \frac{|\phi_{r}(x)-\phi_{r}(y)|^{p}}{|x-y|^{N+sp}} \, dx dy \leq Cr^{-sp} \int_{\B_{Kr}(0)} |u(x)|^{p} \, dx+C K^{-N}
\end{align*}
from which we deduce that
\begin{align*}
& \limsup_{r\rightarrow \infty} \iint_{\R^{2N}} |u(x)|^{p} \frac{|\phi_{r}(x)-\phi_{r}(y)|^{p}}{|x-y|^{N+sp}} \, dx dy \nonumber\\
&\quad \quad =\lim_{K\rightarrow \infty}\limsup_{r\rightarrow \infty} \iint_{\R^{2N}} |u(x)|^{p} \frac{|\phi_{r}(x)-\phi_{r}(y)|^{p}}{|x-y|^{N+sp}} \, dx dy =0.
\end{align*}
\end{proof}

Now we prove the following useful result inspired by \cite{Alv, MeW}.
\begin{lemma}\label{lemVince}
Let $w\in \mathcal{D}^{s, p}(\R^{N})$ and $\{z_{n}\}\subset \mathcal{D}^{s, p}(\R^{N})$ be a sequence such that $z_{n}\rightarrow 0$ a.e. in $\R^{N}$ and $[z_{n}]_{s, p}\leq C$ for any $n\in \mathbb{N}$. Then we have
\begin{align*}
\iint_{\R^{2N}} |\A(z_{n} + w) - \A(z_{n}) - \A(w)|^{p'} dx= o_{n}(1),
\end{align*}
where ${\A(u)=\frac{|u(x)- u(y)|^{p-2}(u(x)- u(y))}{|x-y|^{\frac{N+sp}{p'}}}}$ and $p'= \frac{p}{p-1}$.
\end{lemma}

\begin{proof}
We first consider the case $p\geq 2$. In view of the Mean Value Theorem and Young's inequality, we can see that fixed $\e>0$ there exists $C_{\e}>0$ such that
\begin{equation}\label{abp-2}
||a+b|^{p-2}(a+b)-|a|^{p-2}a|\leq \e |a|^{p-1}+C_{\e}|b|^{p-1} \quad \mbox{ for all } a, b\in \R.
\end{equation}
Taking
$$
a= \frac{z_{n}(x)-z_{n}(y)}{|x-y|^{\frac{N+sp}{p}}}   \quad \mbox{ and } \quad b= \frac{w(x)-w(y)}{|x-y|^{\frac{N+sp}{p}}}
$$
in \eqref{abp-2}, we obtain
\begin{align*}
&\left|\frac{|(z_{n}+w)(x)- (z_{n}+w)(y)|^{p-2}((z_{n}+w)(x)- (z_{n}+w)(y))}{|x-y|^{\frac{N+sp}{p'}}}\right.\\
&\left. -  \frac{|z_{n}(x)- z_{n}(y)|^{p-2}(z_{n}(x)- z_{n}(y))}{|x-y|^{\frac{N+sp}{p'}}} \right| \\
\leq& \e \frac{|z_{n}(x)-z_{n}(y)|^{p-1}}{|x-y|^{\frac{N+sp}{p'}}}+ C_{\e} \frac{|w(x)-w(y)|^{p-1}}{|x-y|^{\frac{N+sp}{p'}}}.
\end{align*}
Let us define $H_{\e, n}:\R^{2N}\rightarrow \R_{+}$ as
\begin{align*}
H_{\e, n} (x, y):= \max \left\{ |\A(z_{n} + w) - \A(z_{n}) - \A(w)|- \e \frac{|z_{n}(x)-z_{n}(y)|^{p-1}}{|x-y|^{\frac{N+sp}{p'}}}, \, 0 \right\}.
\end{align*}
Then we have that $H_{\e, n} \rightarrow 0$ a.e. in $\R^{2N}$ as $n\rightarrow \infty$ and
\begin{align*}
0\leq H_{\e, n}(x, y) \leq C_{\e} \frac{|w(x)-w(y)|^{p-1}}{|x-y|^{\frac{N+sp}{p'}}}\in L^{p'}(\R^{2N}).
\end{align*}
The Dominated Convergence Theorem yields
\begin{align*}
\int_{\R^{2N}} |H_{\e, n} |^{p'} dxdy \rightarrow 0 \quad \mbox{ as } n\rightarrow \infty.
\end{align*}
From the definition of $H_{\e, n}$ we deduce that
\begin{align*}
|\A(z_{n} + w) -\A(z_{n}) - \A(w)|^{p'}\leq C\e \frac{|z_{n}(x)-z_{n}(y)|^{p}}{|x-y|^{N+sp}} + C (H_{\e, n})^{p'}.
\end{align*}
Therefore,
\begin{align*}
\limsup_{n\rightarrow \infty} \iint_{\R^{2N}} |\A(z_{n} + w) - \A(z_{n}) - \A(w)|^{p'} dxdy \leq C \limsup_{n\rightarrow \infty}  \e^{p'} [z_{n}]^{p}_{s, p} \leq C \e^{p'}.
\end{align*}
By the arbitrariness of $\e$ we get the thesis.

Now we deal with the case $1<p<2$.
Using Lemma $3.1$ in \cite{MeW}, we know that
$$
\sup_{c\in \R^{N}, d\neq 0} \left|\frac{|c+d|^{p-2}(c+d)-|c|^{p-2}c}{|d|^{p-1}}\right|<\infty,
$$
so, setting
$$
c=\frac{z_{n}(x)-z_{n}(y)}{|x-y|^{\frac{N+sp}{p}}} \quad \mbox{ and } \quad d=\frac{w(x)-w(y)}{|x-y|^{\frac{N+sp}{p}}},
$$
we can conclude the proof in view of the Dominated Convergence Theorem.
\end{proof}

Let us define the space
\begin{align*}
\X_{\e}:=\left \{u\in W^{s, p}(\R^{N})\cap W^{s, q}(\R^{N}) : \int_{\R^{N}} V(\e x) \left(|u|^{p}+|u|^{q}\right) \,dx<\infty \right \}
\end{align*}
endowed with the norm
\begin{align*}
\|u\|_{\e}= \|u\|_{V, p} + \|u\|_{V, q},
\end{align*}
where
\begin{align*}
\|u\|_{V, t}= \left([u]_{s, t}^{t} + \int_{\R^{N}} V(\e x) |u|^{t}\, dx\right)^{\frac{1}{t}} \quad \mbox{ for all } t>1.
\end{align*}

It is easy to check that the next result holds true.
\begin{lemma}\label{embedding}
The space $\X_{\e}$ is continuously embedded into $W^{s, p}(\R^{N})\cap W^{s, q}(\R^{N})$.
Therefore, $\X_{\e}$ is continuously embedded into $L^{t}(\R^{N})$ for any $t\in [p, \q]$ and compactly embedded into $L^{t}(\B_{R})$, for all $R>0$ and for any $t\in [1, \q)$.
\end{lemma}

\noindent
\begin{lemma}\label{lem6}
If $V_{\infty}=\infty$, the embedding $\X_{\e}\subset L^{m}(\R^{N})$ is compact for any $p\leq m<\q$.
\end{lemma}

\begin{proof}
The space
\begin{align*}
\mathbb{Y}_{\e}= \left\{u\in W^{s, q}(\R^{N}) : \int_{\R^{N}} V(\e x) |u|^{q}dx <\infty\right\}
\end{align*}
endowed with the norm
\begin{align*}
\|u\|^{q}_{\mathbb{Y}_{\e}}= [u]_{s, q}^{q}+ \int_{\R^{N}} V(\e x) |u|^{q}dx
\end{align*}
is compactly embedded in $L^{t}(\R^{N})$ for any $t\in (q, \q)$. Moreover, the space $\X_{\e}$ is continuously embedded in $\mathbb{Y}_{\e}$, therefore, by interpolation, the embedding $\X_{\e}\subset L^{m}(\R^{N})$ is compact for any $p\leq m<\q$.
\end{proof}

The next two results are technical lemmas which will be very useful in this work; their proofs are obtained following the arguments developed by Brezis and Lieb in \cite{BL}.
\begin{lemma}\label{lemPSY}
If $\{u_{n}\}$ is a bounded sequence in $\X_{\e}$, then
\begin{align*}
&[u_{n}-u]^{p}_{s, p}+ [u_{n}-u]^{q}_{s, q}+ \int_{\R^{N}} V(\e x) (|u_{n}-u|^{p} + |u_{n}-u|^{q}) \, dx \\
&\quad = \left( [u_{n}]^{p}_{s, p} + [u_{n}]^{q}_{s, q} + \int_{\R^{N}} V(\e x) (|u_{n}|^{p}+ |u_{n}|^{q}) \, dx\right) \\
&\quad \quad- \left( [u]^{p}_{s, p}+ [u]^{q}_{s, q} + \int_{\R^{N}} V(\e x) (|u|^{p}+ |u|^{q}) \, dx \right)+o_{n}(1).
\end{align*}
\end{lemma}
\begin{proof}
From the Brezis-Lieb Lemma \cite{BL} we know that if $t\in (1, \infty)$ and $\{g_{n}\}\subset L^{t}(\R^{k})$ is a bounded sequence such that $g_{n}\rightarrow g$ a.e. in $\R^{k}$, then we have
\begin{equation}\label{grk}
|g_{n}-g|_{L^{t}(\R^{k})}^{t}= |g_{n}|_{L^{t}(\R^{k})}^{t} - |g|_{L^{t}(\R^{k})}^{t} +o_{n}(1).
\end{equation}
Therefore
\begin{align*}
\int_{\R^{N}} V(\e x) |u_{n}-u|^{p}\, dx= \int_{\R^{N}} V(\e x) |u_{n}|^{p} \, dx- \int_{\R^{N}} V(\e x) |u|^{p}\, dx+ o_{n}(1).
\end{align*}
Taking
\begin{equation*}
g_{n}=\frac{u_{n}(x)-u_{n}(y)}{|x-y|^{\frac{N+sp}{p}}}, \quad  g= \frac{u(x)-u(y)}{|x-y|^{\frac{N+sp}{p}}}, \quad t=p \, \mbox{ and } \, k=2N
\end{equation*}
in \eqref{grk} we obtain
\begin{align*}
[u_{n}-u]^{p}_{s, p}= [u_{n}]^{p}_{s, p}- [u]^{p}_{s, p} + o_{n}(1).
\end{align*}
In similar fashion we can see that
\begin{align*}
[u_{n}-u]^{q}_{s, q}= [u_{n}]^{q}_{s, q}- [u]^{q}_{s, q} + o_{n}(1)
\end{align*}
and
\begin{align*}
\int_{\R^{N}} V(\e x) |u_{n}-u|^{q}\, dx= \int_{\R^{N}} V(\e x) |u_{n}|^{q}\, dx- \int_{\R^{N}} V(\e x) |u|^{q}\, dx+ o_{n}(1).
\end{align*}
This ends the proof of lemma.
\end{proof}

\begin{lemma}\label{lem7}
Let $\{u_{n}\}\subset \X_{\e}$ be a sequence such that $u_{n}\rightharpoonup u$ in $\X_{\e}$. Set $v_{n}=u_{n}-u$.  Then we have
\begin{compactenum}[$(i)$]
\item $\displaystyle{[v_{n}]_{s, p}^{p}+[v_{n}]_{s, q}^{q}= \left([u_{n}]_{s, p}^{p}+[u_{n}]_{s, q}^{q}\right) - \left([u]_{s, p}^{p}+[u]_{s, q}^{q}\right)+o_{n}(1)}$,
\item $\displaystyle\int_{\R^{N}} V(\e x) \left(|v_{n}|^{p}+ |v_{n}|^{q}\right) \, dx$
\item[] $= \int_{\R^{N}} V(\e x) \left(|u_{n}|^{p}+ |u_{n}|^{q}\right) \, dx- \int_{\R^{N}} V(\e x) \left(|u|^{p}+ |u|^{q}\right) \, dx +o_{n}(1)$,
\item $\displaystyle{\int_{\R^{N}} \left(F(v_{n})- F(u_{n})+ F(u)\right) \, dx =o_{n}(1)}$,
\item $\displaystyle{\sup_{\|w\|_{\e}\leq 1} \int_{\R^{N}} |\left(f(v_{n}) - f(u_{n})+ f(u)\right) w|\, dx = o_{n}(1)}$.
\end{compactenum}
\end{lemma}

\begin{proof}
Let us note that the proofs of $(i)$ and $(ii)$ follow by Lemma \ref{lemPSY}.

Now we prove $(iii)$. Let us note that
\begin{align}\label{L1}
F(v_{n})-F(u_{n})=\int_{0}^{1} \frac{d}{dt} F(u_{n}-tu)dt=-\int_{0}^{1} u f(u_{n}-tu)dt,
\end{align}
where $v_{n}=u_{n}-u$.
Then, combining \eqref{L1} and assumptions $(f_{2})$ and $(f_{3})$, we can see that fixed $\delta>0$ there exists $C_{\delta}>0$ such that
$$
|F(v_{n})-F(u_{n})|\leq \delta |u_{n}|^{p-1}|u|+\delta |u|^{p}+C_{\delta}|u_{n}|^{r-1}|u|+C_{\delta}|u|^{r}.
$$
Applying Young's inequality with $\eta>0$, we can deduce that
$$
|F(v_{n})-F(u_{n})|\leq \eta (|u_{n}|^{p}+|u_{n}|^{r})+C_{\eta} (|u|^{p}+|u|^{r})
$$
which implies that
$$
|F(v_{n})-F(u_{n})+F(u)|\leq \eta (|u_{n}|^{p}+|u_{n}|^{r})+C_{\eta} (|u|^{p}+|u|^{r}).
$$
Let
$$
G_{\eta, n}(x)=\max\left\{|F(v_{n})-F(u_{n})+F(u)|-\eta (|u_{n}|^{p}+|u_{n}|^{r}), \, 0\right\}.
$$
Then $G_{\eta, n}\rightarrow 0$ a.e. in $\R^{N}$ as $n\rightarrow \infty$, and $0\leq G_{\eta, n}\leq C_{\eta} (|u|^{p}+|u|^{r})\in L^{1}(\R^{N})$. As a consequence of the Dominated Convergence Theorem, we get
$$
\int_{\R^{N}} G_{\eta, n}(x) \,dx\rightarrow 0 \quad \mbox{ as } n\rightarrow \infty.
$$
On the other hand, from the definition of $G_{\eta, n}$, it follows that
$$
|F(v_{n})-F(u_{n})+F(u)|\leq \eta (|u_{n}|^{p}+|u_{n}|^{r})+C|G_{\eta, n}|
$$
which together with the boundedness of $\{u_{n}\}$ in $L^{p}(\R^{N})\cap L^{r}(\R^{N})$ yields
$$
\limsup_{n\rightarrow \infty} \int_{\R^{N}} |F(v_{n})-F(u_{n})+F(u)| dx\leq C\eta.
$$
From the arbitrariness of $\eta$ we can deduce that $(iii)$ holds true.

Finally, we give the proof of $(iv)$.
For any fixed $\eta>0$, by $(f_{2})$ we can choose $r_{0}=r_{0}(\eta)\in(0, 1)$ such that
\begin{equation}\label{ZZ1}
|f(t)|\leq \eta |t|^{p-1} \quad \mbox{ for } |t|\leq 2r_{0}.
\end{equation}
On the other hand, by $(f_{3})$ we can pick $r_{1}= r_{1}(\eta)>2$ verifying
\begin{equation}\label{ZZ2}
|f(t)|\leq \eta |t|^{\q -1} \quad \mbox{ for } |t|\geq r_{1}-1.
\end{equation}
From the continuity of $f$, there exists $\delta= \delta(\eta)\in (0, r_{0})$ satisfying
\begin{align}\label{ZZ3}
|f(t_{1})- f(t_{2})|\leq r_{0}^{p-1}\eta \quad \mbox{ for } |t_{1}-t_{2}|\leq \delta, \, |t_{1}|, |t_{2}|\leq r_{1}+1.
\end{align}
Moreover, by $(f_{3})$ there exists a positive constant $c=c(\eta)$ such that
\begin{equation}\label{ZZ4}
|f(t)|\leq c(\eta) |t|^{p-1} + \eta |t|^{\q-1} \quad \mbox{ for } t\in \R.
\end{equation}
In what follows we estimate the following term:
$$
\int_{\R^{N}\setminus \B_{R}(0)} |f(u_{n}-u)- f(u_{n})-f(u)||w|\, dx.
$$
Using \eqref{ZZ4} and $u\in L^{p}(\R^{N})\cap L^{\q}(\R^{N})$ we can find $R=R(\eta)>0$ such that
\begin{align*}
\int_{\R^{N}\setminus \B_{R}(0)} |f(u) w|\, dx\leq& c\left(\int_{\R^{N}\setminus \B_{R}(0)} |u|^{\q}\, dx \right)^{\frac{\q-1}{\q}} |w|_{\q}\\
&+ c \left(\int_{\R^{N}\setminus \B_{R}(0)} |u|^{p}\, dx \right)^{\frac{p-1}{p}} |w|_{p}\\
\leq& c \eta \|w\|_{s, q} + c\eta \|w\|_{s, p}\leq c\eta \|w\|_{\e}.
\end{align*}
Set $A_{n}:=\{x\in \R^{N}\setminus \B_{R}(0) : |u_{n}(x)|\leq r_{0}\}$. In view of \eqref{ZZ1} and applying H\"older inequality we get
\begin{align}\label{ZZ5}
\int_{A_{n}\cap \{|u|\leq \delta\}} |f(u_{n})- f(u_{n}-u)||w|\, dx \leq \eta (|u_{n}|_{p}^{p-1} + |u_{n}-u|_{p}^{p-1})|w|_{p} \leq c \eta \|w\|_{\e}.
\end{align}
Let $B_{n}:= \{x\in \R^{N} \setminus \B_{R}(0) : |u_{n}(x)|\geq r_{1}\}$. Then \eqref{ZZ2} and H\"older inequality yield
\begin{align}\label{ZZ6}
\int_{B_{n}\cap \{|u|\leq \delta\}} |f(u_{n})\!-\! f(u_{n}\!-\!u)| |w| dx \leq \eta (|u_{n}|_{\q}^{\q-1} \!+\! |u_{n}\!-\!u|_{\q}^{\q-1}) |w|_{\q} \leq c\eta \|w\|_{\e}.
\end{align}
Finally, define $C_{n}:= \{x\in \R^{N} \setminus \B_{R}(0) : r_{0}\leq |u_{n}(x)|\leq r_{1}\}$. Since $u_{n}\in W^{s, p}(\R^{N})$ we know that $|C_{n}|<\infty$. Then \eqref{ZZ3} gives
\begin{align}\label{ZZ7}
\int_{C_{n}\cap \{|u|\leq \delta\}} |f(u_{n})- f(u_{n}-u)| |w| \, dx \leq r_{0}^{p-1}\eta |w|_{p} |C_{n}|^{\frac{p-1}{p}}\leq \eta |u_{n}|_{p} |w|_{p} \leq c \eta \|w\|_{\e}.
\end{align}
Putting together \eqref{ZZ5}, \eqref{ZZ6} and \eqref{ZZ7} we have that
\begin{align}\label{ZZ8}
\int_{(\R^{N}\setminus \B_{R}(0))\cap \{|u|\leq \delta\}} |f(u_{n})- f(u_{n}-u)| |w| \, dx \leq c \eta \|w\|_{\e} \quad \mbox{ for all } n\in \mathbb{N}.
\end{align}
Now, we note that \eqref{ZZ4} implies
\begin{align*}
|f(u_{n})- f(u_{n}-u)|\leq \eta (|u_{n}|^{\q-1} + |u_{n}-u|^{\q-1}) + c(\eta) (|u_{n}|^{p-1}+ |u_{n}-u|^{p-1}),
\end{align*}
so we can see that
\begin{align*}
&\int_{(\R^{N}\setminus \B_{R}(0)) \cap \{|u|\geq \delta\}} |f(u_{n})- f(u_{n}-u)||w|\, dx \\
&\leq \int_{(\R^{N}\setminus \B_{R}(0)) \cap \{|u|\geq \delta\}} \left[ \eta (|u_{n}|^{\q-1} + |u_{n}-u|^{\q-1}) |w|\right.\\
&\left.\hspace{3.6cm}+ c(\eta) (|u_{n}|^{p-1}+ |u_{n}-u|^{p-1})|w| \right]\, dx\\
&\leq c\eta \|w\|_{\e} +  \int_{(\R^{N}\setminus \B_{R}(0)) \cap \{|u|\geq \delta\}} c(\eta) (|u_{n}|^{p-1}+ |u_{n}-u|^{p-1})|w|\, dx.
\end{align*}
Since $u\in W^{s, p}(\R^{N})$, we get $|(\R^{N}\setminus \B_{R}(0)) \cap \{|u|\geq \delta\}|\ri 0$ as $R\ri \infty$. Then choosing $R=R(\eta)$ large enough we can infer
\begin{align*}
&\int_{(\R^{N}\setminus \B_{R}(0)) \cap \{|u|\geq \delta\}} c(\eta) (|u_{n}|^{p-1}+ |u_{n}-u|^{p-1})|w|\, dx \\
&\quad \leq c(\eta) (|u_{n}|_{\q}^{p-1}+ |u_{n}-u|_{\q}^{p-1}) \,|w|_{\q} \, |(\R^{N}\setminus \B_{R}(0)) \cap \{|u|\geq \delta\}|^{\frac{\q-p}{p}}\leq \eta \|w\|_{\e},
\end{align*}
where we used the generalized H\"older inequality. Therefore
\begin{align*}
\int_{(\R^{N}\setminus \B_{R}(0)) \cap \{|u|\geq \delta\}} |f(u_{n})- f(u_{n}-u)||w|\, dx \leq c \eta \|w\|_{\e} \quad \mbox{ for all } n\in \mathbb{N},
\end{align*}
which combined with \eqref{ZZ8} yields
\begin{align}\label{ZZ17}
\int_{\R^{N}\setminus \B_{R}(0)} |f(u_{n})-f(u)- f(u_{n}-u)||w|\, dx \leq c \eta \|w\|_{\e} \quad \mbox{ for all } n\in \mathbb{N}.
\end{align}
Now, recalling that $u_{n}\rightharpoonup u$ in $W^{s, p}(\R^{N})$ we may assume that, up to a subsequence, $u_{n}\ri u$ strongly in $L^{p}(\B_{R}(0))$ and there exists $h\in L^{p}(\B_{R}(0))$ such that $|u_{n}(x)|, |u(x)|\leq |h(x)|$ a. e. $x\in \B_{R}(0)$.

It is clear that
\begin{align}\label{ZZ18}
\int_{\B_{R}(0)} |f(u_{n}-u)||w|\, dx \leq c\eta \|w\|_{\e}
\end{align}
provided that $n$ is big enough. Let us define $D_{n}:=\{x\in \B_{R}(0) : |u_{n}(x) - u(x)|\geq 1\}$. Thus
\begin{align*}
&\int_{D_{n}} |f(u_{n})- f(u)||w|\, dx\\
 &\leq \int_{D_{n}} \left( c(\eta) (|u|^{p-1}+ |u_{n}|^{p-1}) + \eta (|u_{n}|^{\q-1} + |u|^{\q-1})\right) |w|\, dx \\
&\leq c\eta \|w\|_{\e} + 2c(\eta) \int_{D_{n}} |h|^{p-1}|w|\, dx \\
&\leq c\eta \|w\|_{\e} + 2c(\eta) \left( \int_{D_{n}} |h|^{p} \, dx\right)^{\frac{p-1}{p}} |w|_{p}.
\end{align*}
Observing that $|D_{n}|\ri 0$ as $n\ri \infty$, we can deduce that
\begin{align}\label{ZZ19}
\int_{D_{n}} |f(u_{n})- f(u)||w| \, dx \leq c\eta \|w\|_{\e}.
\end{align}
Since $u\in W^{s, p}(\R^{N})$, we know that $|\{|u|\geq L\}|\ri 0$ as $L\ri \infty$, so there exists $L= L(\eta)>0$ such that for all $n$
\begin{align}\label{ZZ20}
&\int_{(\B_{R}(0)\setminus D_{n})\cap \{|u|\geq L\}} |f(u_{n})- f(u)||w|\, dx \nonumber \\
&  \leq \int_{(\B_{R}(0)\setminus D_{n})\cap \{|u|\geq L\}} \left[\eta (|u_{n}|^{\q-1} + |u|^{\q-1})|w| + c(\eta) (|u_{n}|^{p-1}+ |u|^{p-1})|w| \right]\, dx \nonumber  \\
&  \leq c\eta \|w\|_{\e}+ c(\eta) (|u_{n}|_{\q}^{p-1} + |u|_{\q}^{p-1}) \,|w|_{\q}\, |(\B_{R}(0)\setminus D_{n})\cap \{|u|\geq L\}|^{\frac{\q-p}{p}} \nonumber \\
&  \leq c \eta \|w\|_{\e}.
\end{align}
On the other hand, by the Dominated Convergence Theorem we can infer
\begin{align*}
\int_{(\B_{R}(0)\setminus D_{n})\cap \{|u|\leq L\}} |f(u_{n})- f(u)|^{p}\, dx \ri 0 \quad  \mbox{ as } n\ri \infty.
\end{align*}
As a consequence
\begin{align}\label{ZZ21}
\int_{(\B_{R}(0)\setminus D_{n})\cap \{|u|\leq L\}} |f(u_{n})- f(u)| |w|\, dx\leq c \eta \|w\|_{\e}
\end{align}
for $n$ large enough. Putting together \eqref{ZZ19}, \eqref{ZZ20} and \eqref{ZZ21} we have
\begin{align*}
\int_{\B_{R}(0)} |f(u_{n})- f(u)| |w|\, dx \leq c \eta \|w\|_{\e}.
\end{align*}
This and \eqref{ZZ18} yield
\begin{align}\label{ZZ22}
\int_{\B_{R}(0)} |f(u_{n})- f(u)- f(u_{n}-u)| |w| \, dx \leq c \eta \|w\|_{\e}.
\end{align}
Taking into account \eqref{ZZ17} and \eqref{ZZ22} we can conclude that for $n$ large enough
\begin{align*}
\int_{\R^{N}} |f(u_{n})- f(u)- f(u_{n}-u)| |w|\, dx \leq c \eta \|w\|_{\e}.
\end{align*}
\end{proof}

\section{The autonomous problem}\label{Sect3}

In this section we consider the autonomous problem associated to \eqref{P}:
\begin{equation}\label{Pmu}
(-\Delta)_{p}^{s}u+ (-\Delta)_{q}^{s}u + \mu (|u|^{p-2}u + |u|^{q-2}u)= f(u) \mbox{ in } \R^{N}
\end{equation}
for all $\mu>0$.

The corresponding functional is given by
\begin{equation*}
\J_{\mu}(u)= \frac{1}{p}[u]_{s, p}^{p}+\frac{1}{q}[u]_{s, q}^{q} + \mu \left[\frac{1}{p} |u|_{p}^{p} +\frac{1}{q}|u|_{q}^{q}\right]-\int_{\R^{N}} F(u) \, dx
\end{equation*}
which is well-defined on the space $\X_{\mu}=W^{s, p}(\R^{N})\cap W^{s, q}(\R^{N})$ endowed with the norm
\begin{align*}
\|u\|_{\mu}= \|u\|_{\mu, p}+ \|u\|_{\mu, q},
\end{align*}
where
\begin{align*}
\|u\|_{\mu, t}= \left([u]_{s,t}^{t}+\mu |u|_{t}^{t} \,\right)^{\frac{1}{t}} \quad \mbox{ for all } t> 1.
\end{align*}
It is easy to check that $\J_{\mu}\in \C^{1}(\X_{\mu}, \R)$ and its differential is given by
\begin{align*}
\langle \J'_{\mu}(u), \varphi \rangle &= \iint_{\R^{2N}} \frac{|u(x)-u(y)|^{p-2}(u(x)-u(y))(\varphi(x)- \varphi(y))}{|x-y|^{N+sp}}\, dxdy \\
&\quad + \iint_{\R^{2N}} \frac{|u(x)-u(y)|^{q-2}(u(x)-u(y))(\varphi(x)- \varphi(y))}{|x-y|^{N+sq}}\, dxdy \\
&\quad +\mu \left[\int_{\R^{N}} |u|^{p-2}u\,\varphi \, dx+ \int_{\R^{N}} |u|^{q-2}u\,\varphi \, dx \right]  - \int_{\R^{N}} f(u)\varphi \, dx
\end{align*}
for any $u, \varphi \in \X_{\mu}$. Let us define the Nehari manifold associated to $\J_{\mu}$
\begin{equation*}
\N_{\mu} = \{u\in \X_{\mu}\setminus \{0\} : \langle \J'_{\mu}(u), u\rangle =0\}.
\end{equation*}
Now we prove that $\J_{\mu}$ possesses a mountain pass geometry \cite{AR}.
\begin{lemma}\label{lem1}
The functional $\J_{\mu}$ satisfies the following conditions:
\begin{compactenum}[(i)]
\item there exist $\alpha, \rho>0$ such that $\J_{\mu}(u)\geq \alpha$ with $\|u\|_{\mu}=\rho$;
\item there exists $e\in \X_{\mu}$ with $\|e\|_{\mu}>\rho$ such that $\J_{\mu}(e)<0$.
\end{compactenum}
\end{lemma}

\begin{proof}
$(i)$ From assumptions $(f_{2})$ and $(f_{3})$, for any $\xi>0$ there exists $C_{\xi}$ such that
\begin{equation}\label{growthf}
|f(u)|\leq \xi |u|^{p-1} + C_{\xi}|u|^{r-1}.
\end{equation}
Therefore,
\begin{align*}
\J_{\mu}(u)&\geq \frac{1}{p}[u]_{s, p}^{p}+\frac{1}{q}[u]_{s, q}^{q} + \mu \left[\frac{1}{p} |u|_{p}^{p} +\frac{1}{q}|u|_{q}^{q}\right]-\frac{\xi}{p}|u|_{p}^{p} -\frac{C_{\xi}}{r}|u|_{r}^{r} \\
&= \frac{1}{p}[u]_{s, p}^{p}+\frac{1}{q}[u]_{s, q}^{q} +\frac{1}{p}(\mu -\xi)|u|_{p}^{p} + \frac{\mu}{q}|u|_{q}^{q} -\frac{C_{\xi}}{r}|u|_{r}^{r}.
\end{align*}
Choosing $\|u\|_{\mu}=\rho \in (0, 1)$, by Sobolev embedding we get
\begin{align*}
\J_{\mu}(u)\geq C_{1} \|u\|_{\mu}^{q}-\frac{C_{\xi}}{r}|u|_{r}^{r}\geq C_{1} \|u\|_{\mu}^{q}-C_{2} \|u\|_{\mu}^{r}.
\end{align*}
Taking into account that $r>q$, there is $\alpha>0$ such that $\J_{\mu}(u)\geq \alpha>0$ with $\|u\|_{\mu}=\rho$.

\noindent
$(ii)$
Fix $\varphi \in \C^{\infty}_{c}(\R^{N})$ such that $\varphi> 0$ in $\R^{N}$. Then using $(f_{4})$ and Fatou's Lemma we can deduce that
\begin{align*}
\frac{\J_{\mu}(t\varphi)}{\|t\varphi\|_{\mu}^{q}} &\leq \frac{1}{p}\frac{\|\varphi\|_{\mu, p}^{p}}{t^{q-p} \|\varphi\|_{\mu}^{q}}+\frac{1}{q}\frac{\|\varphi\|_{\mu, q}^{q}}{\|\varphi\|_{\mu}^{q}} -\int_{\supp \varphi} \frac{F(t\varphi)}{\|t\varphi\|_{\mu}^{q}} \, dx\\
&\leq \frac{1}{p}\frac{1}{t^{q-p}}+\frac{1}{q} -\int_{\supp \varphi} \frac{F(t\varphi)}{(t\varphi)^{q}} \left(\frac{\varphi}{\|\varphi\|_{\mu}}\right)^{q} \, dx\ri -\infty \,\, \mbox{ as } t\rightarrow \infty.
\end{align*}
\end{proof}

As a consequence of the mountain pass theorem without $(PS)$ condition (see \cite{W}) we can find a $(PS)_{c_{\mu}}$ sequence $\{u_{n}\}\subset \X_{\mu}$, that is
\begin{equation*}
\J_{\mu}(u_{n})\ri c_{\mu} \quad \mbox{ and } \quad \J_{\mu}'(u_{n})\ri 0,
\end{equation*}
where
\begin{align*}
c_{\mu}= \inf_{\gamma\in \Gamma} \max_{t\in [0, 1]} \J_{\mu}(\gamma(t))
\end{align*}
and
\begin{align*}
\Gamma=\{\gamma\in \C([0, 1], \X_{\mu}) \, : \, \gamma(0)=0, \, \J_{\mu}(\gamma(1))<0\}.
\end{align*}
In what follows we give a very useful characterization of $c_{\mu}$.

\begin{lemma}\label{lem2}
Assume that $(f_{1})$-$(f_{5})$ hold. Then, for each $u\in \X_{\mu}$ with $u\neq 0$, there exists a unique $t_{0}= t_{0}(u)>0$ such that $t_{0}u\in \N_{\mu}$ and $\J_{\mu}(t_{0}u)= \max_{t\geq 0} \J_{\mu}(tu)$. Moreover
\begin{align*}
c_{\mu}= \inf_{u\in \X_{\mu}\setminus \{0\}} \max_{t\geq 0} \J_{\mu}(tu)= \inf_{u\in \N_{\mu}} \J_{\mu}(u).
\end{align*}
\end{lemma}

\begin{proof}
Let $u\in \X_{\mu}\setminus\{0\}$ and define $h(t):= \J_{\mu}(tu)$. Then, from the arguments in Lemma \ref{lem1}, we know that there exists $t_{0}>0$ such that $h'(t_{0})=0$ and $t_{0}u\in \N_{\mu}$. Let us note that if $u\in \N_{\mu}$, then hypothesis $(f_{1})$ ensure that $u^{+}=\max\{u, 0\}\neq 0$.

Now, we aim to prove that $t_{0}$ is the unique critical point of $h$. Arguing by contradiction, let us take positive $t_{1}$ and $t_{2}$ such that $t_{1}u, t_{2}u\in \N_{\mu}$. Then we have
\begin{align*}
t_{1}^{p-q}[u]_{s, p}^{p} +[u]_{s, q}^{q}+ \mu \left[ t_{1}^{p-q} |u|_{p}^{p} + |u|_{q}^{q}\right] =\int_{\R^{N}} \frac{f(t_{1}u)}{(t_{1}u)^{q-1} }u^{q}\, dx
\end{align*}
and
\begin{align*}
t_{2}^{p-q}[u]_{s, p}^{p} + [u]_{s, q}^{q}+ \mu \left[t_{2}^{p-q}|u|_{p}^{p} + |u|_{q}^{q}\right] =\int_{\R^{N}} \frac{f(t_{2}u)}{(t_{2}u)^{q-1} }u^{q}\, dx.
\end{align*}
Subtracting terms by terms the above equalities we have
\begin{align*}
(t_{1}^{p-q}-t_{2}^{p-q})[u]_{s, p}^{p} + \mu (t_{1}^{p-q}-t_{2}^{p-q}) |u|_{p}^{p} = \int_{\R^{N}} \left[\frac{f(t_{1}u)}{(t_{1}u)^{q-1} } - \frac{f(t_{2}u)}{(t_{2}u)^{q-1} } \right] u^{q} dx,
\end{align*}
which allows us to deduce a contradiction. Indeed, if $t_{1}<t_{2}$, taking into account that $p<q$ and using $(f_{5})$ we get
\begin{align*}
0<(t_{1}^{p-q}-t_{2}^{p-q})[u]_{s, p}^{p} \!+\! \mu (t_{1}^{p-q}-t_{2}^{p-q}) |u|_{p}^{p} = \int_{\R^{N}} \left[\frac{f(t_{1}u)}{(t_{1}u)^{q-1} } \!-\! \frac{f(t_{2}u)}{(t_{2}u)^{q-1} } \right] u^{q} dx<0.
\end{align*}
\end{proof}

\begin{lemma}\label{lem3}
Let $\{u_{n}\}\subset \N_{\mu}$ be such that $\J_{\mu}(u_{n})\ri c$. Then, $\{u_{n}\}$ is bounded in $\X_{\mu}$.
\end{lemma}

\begin{proof}
Assume by contradiction that $\|u_{n}\|_{\mu}\ri \infty$ for some subsequence.
Set $v_{n}= \frac{u_{n}}{\|u_{n}\|_{\mu}}$. Then $\|v_{n}\|_{\mu}=1$ for any $n\in \mathbb{N}$.
Moreover, for each $\rho>0$ it holds that
\begin{align}\label{new1}
\lim_{n\ri \infty} \sup_{y\in \R^{N}} \int_{\B_{\rho}(y)} |v_{n}|^{q} \, dx=0.
\end{align}
If \eqref{new1} does not hold, then for some $\rho>0$ there exist $\delta>0$ and a sequence $\{y_{n}\}\subset \R^{N}$ such that
\begin{align*}
\int_{\B_{\rho}(y_{n})} |v_{n}|^{q} \, dx\geq \delta>0.
\end{align*}
Let $\tilde{v}_{n}=v_{n}(\cdot+ y_{n})$. Since $\{\tilde{v}_{n}\}$ is bounded in $\X_{\mu}$, we may assume that, up to a subsequence,
\begin{align*}
\tilde{v}_{n} \rightharpoonup \tilde{v} &\mbox{ in } \X_{\mu}\\
\tilde{v}_{n} \ri \tilde{v} &\mbox{ in } L^{t}_{loc}(\R^{N}) \mbox{ for any } t\in[p, \q).
\end{align*}
Then, $\tilde{v}\not \equiv 0$ in view of
\begin{align*}
\int_{\B_{\rho}(0)} |\tilde{v}|^{q} dx= \lim_{n\ri \infty} \int_{\B_{\rho}(0)} |\tilde{v}_{n}|^{q} dx = \lim_{n\ri \infty} \int_{\B_{\rho}(y_{n})} |v_{n}|^{q} dx \geq \delta.
\end{align*}
Set $\tilde{u}_{n}= \|u_{n}\|_{\mu} \tilde{v}_{n}$. From $(f_{5})$ we have
\begin{align*}
\frac{F(\tilde{u}_{n})}{|\tilde{u}_{n}|^{q}} |\tilde{v}_{n}|^{q} \ri \infty \quad \mbox{ in } \Omega= \{x\in \R^{N} : \tilde{v}(x) \neq 0\}.
\end{align*}
Hence, recalling that $q>p$, $\|u_{n}\|_{\mu}\ri \infty$,  $\J_{\mu}(u_{n})= c+o_{n}(1)$ and applying Fatou's Lemma we obtain
\begin{align*}
\frac{1}{q}+o_{n}(1)&=\frac{1}{p}\|u_{n}\|_{\mu}^{p-q}+\frac{1}{q}-\frac{c+o_{n}(1)}{\|u_{n}\|_{\mu}^{q}}\\
&=\frac{1}{p} \frac{\|u_{n}\|_{\mu}^{p}}{\|u_{n}\|_{\mu}^{q}} + \frac{1}{q}\frac{\|u_{n}\|_{\mu}^{q}}{\|u_{n}\|_{\mu}^{q}} - \frac{c+o_{n}(1)}{\|u_{n}\|_{\mu}^{q}}\\
&\geq \frac{1}{p} \frac{\|u_{n}\|_{\mu, p}^{p}}{\|u_{n}\|_{\mu}^{q}} + \frac{1}{q}\frac{\|u_{n}\|_{\mu, q}^{q}}{\|u_{n}\|_{\mu}^{q}} - \frac{c+o_{n}(1)}{\|u_{n}\|_{\mu}^{q}}\\
&= \int_{\R^{N}} \frac{F(u_{n})}{\|u_{n}\|_{\mu}^{q}}\, dx=\int_{\R^{N}} \frac{F(\tilde{u}_{n})}{\|u_{n}\|_{\mu}^{q}}\, dx\geq \int_{\Omega} \frac{F(\tilde{u}_{n})}{|\tilde{u}_{n}|^{q}} |\tilde{v}_{n}|^{q}\, dx \ri \infty,
\end{align*}
and this is impossible. Thus, \eqref{new1} holds true and by Lemma \ref{lemVT} we have that $u_{n}\ri 0$ in $L^{m}(\R^{N})$ for any $m\in (q, \q)$. From assumptions $(f_{2})$ and $(f_{3})$, for any $\alpha>1$ and $\tau>0$, there exists a positive constant $C_{\tau}$ such that
\begin{equation}\label{F1}
|F(\alpha t)|\leq \tau |\alpha t|^{p} +C_{\tau} |\alpha t|^{r}
\end{equation}
from which
\begin{align*}
\limsup_{n\ri \infty} \int_{\R^{N}} |F(\alpha v_{n})|\, dx \leq \limsup_{n\ri \infty} \int_{\R^{N}} ( \tau |\alpha v_{n}|^{p} +C_{\tau} |\alpha v_{n}|^{r})\, dx\leq \tau \alpha^{p},
\end{align*}
and letting the limit as $\tau\ri 0$ we can infer that
\begin{align*}
\lim_{n\ri \infty} \int_{\R^{N}} F(\alpha v_{n})\, dx=0.
\end{align*}
Now, for any $R>0$, we note that $\frac{R}{\|u_{n}\|_{\mu}}\in (0, 1)$ for $n$ large.
Let us observe that from $p<q$ and $\|u_{n}\|_{\mu ,p}\leq \|u_{n}\|_{\mu}$ it follows that
$$
\frac{\|u_{n}\|_{\mu,p}^{p}}{\|u_{n}\|_{\mu}^{p}}\geq\frac{\|u_{n}\|_{\mu,p}^{q}}{\|u_{n}\|_{\mu}^{q}}
$$
and for all $R>1$ we also have $R^{q}>R^{p}$. Using these inequalities and $a^{q}+b^{q}\geq C_{q}(a+b)^{q}$ for all $a, b\geq 0$ and $q>1$, we can infer that
\begin{align*}
\J_{\mu}(u_{n})&= \max_{t\geq 0} \J_{\mu}(tu_{n}) \geq \J_{\mu}\left( \frac{R}{\|u_{n}\|_{\mu}} u_{n} \right) = \J_{\mu}(Rv_{n}) \\
&= \frac{R^{p}}{p} \frac{\|u_{n}\|_{\mu, p}^{p}}{\|u_{n}\|_{\mu}^{p}} + \frac{R^{q}}{q} \frac{\|u_{n}\|_{\mu, q}^{q}}{\|u_{n}\|_{\mu}^{q}}- \int_{\R^{N}} F(Rv_{n})dx \\
&\geq  \frac{R^{p}}{q} \left(\frac{\|u_{n}\|_{\mu, p}^{q}+\|u_{n}\|_{\mu, q}^{q}}{\|u_{n}\|_{\mu}^{q}}\right)+o_{n}(1) \\
&\geq \frac{R^{p}}{q} C_{q}+o_{n}(1) \quad \forall R>1.
\end{align*}
Taking the limit as $R\ri \infty$ we can deduce that $\J_{\mu}(u_{n})\ri \infty$ which gives a contradiction.
\end{proof}

Now we prove the next technical lemma which is crucial to show that a $(PS)$ sequence of $\J_{\mu}$ on $\N_{\mu}$ is a $(PS)$ sequence of $\J_{\mu}$ in $\X_{\mu}$.

\begin{proposition}\label{prop1}
Let $\{u_{n}\}\subset \N_{\mu}$ be such that $\J_{\mu}(u_{n})\ri c$ with $u_{n}\rightharpoonup 0$ in $\X_{\mu}$. Then, one of the following alternatives occurs:
\begin{compactenum}[$(a)$]
\item $u_{n}\ri 0$ in $\X_{\mu}$;
\item there are a sequence $\{y_{n}\}\subset \R^{N}$ and constants $R, \beta>0$ such that
\begin{align*}
\liminf_{n\ri \infty} \int_{\B_{R}(y_{n})} |u_{n}|^{q} dx \geq \beta>0.
\end{align*}
\end{compactenum}
\end{proposition}

\begin{proof}
Assume that $(b)$ does not hold true. Then, for any $R>0$ it holds
\begin{align*}
\lim_{n\ri \infty} \sup_{y\in \R^{N}} \int_{\B_{R}(y)} |u_{n}|^{q} dx =0.
\end{align*}
Since $\{u_{n}\}$ is bounded in $\X_{\mu}$, from Lemma \ref{lemVT} it follows that
\begin{align}\label{un0r}
u_{n}\ri 0 \mbox{ in } L^{t}(\R^{N}) \mbox{ for any } t\in (q, \q).
\end{align}
Fix $\xi \in (0, \mu)$. Then, taking into account that $\{u_{n}\}\subset \N_{\mu}$ and \eqref{growthf} we have
\begin{align*}
0&= \langle \J'_{\mu}(u_{n}), u_{n} \rangle \\
&\geq [u_{n}]_{s, p}^{p}+ [u_{n}]_{s, q}^{q} + \mu \left[ |u_{n}|_{p}^{p} + |u_{n}|_{q}^{q}\right] - \xi |u_{n}|_{p}^{p} - C_{\xi} |u_{n}|_{r}^{r}\\
&\geq C_{1}\|u_{n}\|_{s, p}^{p}+C_{2} \|u_{n}\|_{s, q}^{q} -C_{3}|u_{n}|_{r}^{r},
\end{align*}
and in view of \eqref{un0r} we have that $\|u_{n}\|_{\mu}\ri 0$.
\end{proof}

\begin{corollary}\label{cor2.4}
Let $\{u_{n}\}\subset \N_{\mu}$ be a $(PS)_{c}$ sequence for $\J_{\mu}$. Then $\{u_{n}\}$ is a $(PS)_{c}$ sequence for $\J_{\mu}$ in $\X_{\mu}$.
\end{corollary}

\begin{proof}
Let $\{u_{n}\}\subset \N_{\mu}$ be such that $\J_{\mu}(u_{n})\ri c$ and $\|\J'_{\mu}(u_{n})\|_{*}=o_{n}(1)$, where $\|\J'_{\mu}(u)\|_{*}$ denotes the norm of the derivative of the restriction of $\J_{\mu}$ to $\N_{\mu}$ at $u$. Then, there exists $\{\la_{n}\}\subset \R$ such that
\begin{align}\label{nonumb2a}
\J'_{\mu}(u_{n})= \la_{n} \I'_{\mu}(u_{n})+ o_{n}(1),
\end{align}
where $\I_{\mu}: \X_{\mu}\ri \R$ is defined as
\begin{align*}
\I_{\mu}(u)= \|u\|_{\mu, p}^{p}+ \|u\|_{\mu, q}^{q}- \int_{\R^{N}} f(u)u\, dx.
\end{align*}
From $\{u_{n}\}\subset \N_{\mu}$ and $(f_{5})$ it follows that
\begin{align}\label{nonumb1a}
 \langle \I'_{\mu}(u_{n}), u_{n} \rangle
=& [u_{n}]_{s, p}^{p}+ [u_{n}]_{s, q}^{q} +\mu \int_{\R^{N}} (p|u_{n}|^{p}+q|u_{n}|^{q}) \, dx \nonumber\\
& - \int_{\R^{N}} f(u_{n})u_{n}\, dx - \int_{\R^{N}} f'(u_{n})|u_{n}|^{2}\, dx\nonumber \\
\leq& - \int_{\R^{N}} \left( f'(u_{n})|u_{n}|^{2}-(q-1) f(u_{n})u_{n} \right)\, dx.
\end{align}
Since $\{u_{n}\}$ is bounded and $\|u_{n}\|_{\mu}\not\ri 0$, by Proposition \ref{prop1} there exists a sequence $\{y_{n}\}\subset \R^{N}$ such that $\tilde{u}_{n}= u_{n}(\cdot +y_{n})$ is bounded in $\X_{\mu}$ and $\tilde{u}_{n}\rightharpoonup \tilde{u}$ in $\X$ for some $\tilde{u}\neq 0$. Consequently, there exists $\Omega \subset \R^{N}$ with positive measure such that $\tilde{u}>0$ in $\Omega$.

Suppose by contradiction that
$$
\limsup_{n\ri \infty} \, \langle \I'_{\mu}(u_{n}), u_{n}\rangle=0.
$$
Then, using \eqref{nonumb1a}, $(f_{5})$ and Fatou's Lemma we have
\begin{align*}
0\leq - \int_{\Omega} \left( f'(\tilde{u})|\tilde{u}|^{2}-(q-1) f(\tilde{u})\tilde{u} \right)\, dx < 0.
\end{align*}
which gives a contradiction. Hence $\limsup_{n\ri \infty} \langle \I'_{\mu}(u_{n}), u_{n}\rangle<0$ and, as a consequence, $\la_{n}=o_{n}(1)$. This and \eqref{nonumb2a} imply that $\{u_{n}\}$ is a $(PS)_{c}$ sequence for $\J_{\mu}$ in $\X_{\mu}$.
\end{proof}

We end this section giving the proof of the existence of a nonnegative ground state solution for autonomous problem \eqref{Pmu}.

\begin{proposition}\label{prop2}
Assume that $(f_{1})$-$(f_{5})$ hold. Then, problem \eqref{Pmu} has a nonnegative ground state solution.
\end{proposition}

\begin{proof}
Applying the Ekeland variational principle \cite{Ekeland}, there exist sequences $\{u_{n}\}\subset \N_{\mu}$ and $\{\la_{n}\}\subset \R$ such that
\begin{equation*}
\J_{\mu}(u_{n})\ri c_{\mu} \quad \mbox{ and } \quad \|\J'_{\mu}(u_{n})- \la_{n} \I'_{\mu}(u_{n})\|_{\X'}\ri 0
\end{equation*}
where $\I(u)= \langle \J'(u), u\rangle$ for any $u\in \X_{\mu}$.

Following the proof of Corollary \ref{cor2.4}, we can see that $\la_{n}=o_{n}(1)$, so $\{u_{n}\}\subset \N_{\mu}$ is a $(PS)_{c_{\mu}}$ sequence. From Lemma \ref{lem3} $\{u_{n}\}$ is bounded in $\X_{\mu}$, which is a reflexive space, so we may assume that $u_{n}\rightharpoonup u$ in $\X_{\mu}$ for some $u\in \X_{\mu}$.

In what follows, we show that $\J_{\mu}'(u)=0$. Consider the sequence
$$
h_{n}(x,y)=\frac{|u_{n}(x)-u_{n}(y)|^{p-2}(u_{n}(x)-u_{n}(y))}{|x-y|^{\frac{N+sp}{p'}}},
$$
and let
$$
h(x,y)=\frac{|u(x)-u(y)|^{p-2}(u(x)-u(y))}{|x-y|^{\frac{N+sp}{p'}}},
$$
where $p'=\frac{p}{p-1}$. It is easy to check that $\{h_{n}\}$ is a bounded sequence in $L^{p'}(\R^{2N})$ with $h_{n}\rightarrow h$ a.e. in $\R^{2N}$. Since $L^{p'}(\R^{2N})$ is a reflexive space, there exists a subsequence, still denoted by $\{h_{n}\}$, such that $h_{n}\rightharpoonup h$ in $L^{p'}(\R^{2N})$, that is
$$
\iint_{\R^{2N}} h_{n}(x,y) g(x,y) dx dy\rightarrow \iint_{\R^{2N}} h(x,y) g(x,y) dx dy \quad \forall g\in L^{p}(\R^{2N}).
$$
Then, for any $\phi\in \X_{\mu}$, we know that
$$
g(x,y)=\frac{(\phi(x)-\phi(y))}{|x-y|^{\frac{N+sp}{p}}}\in L^{p}(\R^{2N}),
$$
and we can see that
\begin{align*}
\iint_{\R^{2N}} &\frac{|u_{n}(x)-u_{n}(y)|^{p-2}(u_{n}(x)-u_{n}(y))(\phi(x)-\phi(y))}{|x-y|^{N+sp}} dx dy \\
&\rightarrow \iint_{\R^{2N}} \frac{|u(x)-u(y)|^{p-2}(u(x)-u(y))(\phi(x)-\phi(y))}{|x-y|^{N+sp}} dx dy.
\end{align*}
In a similar way we can prove that
\begin{align*}
\iint_{\R^{2N}} &\frac{|u_{n}(x)-u_{n}(y)|^{q-2}(u_{n}(x)-u_{n}(y))(\phi(x)-\phi(y))}{|x-y|^{N+sq}} dx dy \\
&\rightarrow \iint_{\R^{2N}} \frac{|u(x)-u(y)|^{q-2}(u(x)-u(y))(\phi(x)-\phi(y))}{|x-y|^{N+sq}} dx dy.
\end{align*}
Since it is clear that
$$
\int_{\R^{N}} |u_{n}|^{p-2}u_{n}\phi\, dx\rightarrow \int_{\R^{N}} |u|^{p-2}u\phi\, dx,
$$
$$
\int_{\R^{N}} |u_{n}|^{q-2}u_{n}\phi\, dx\rightarrow \int_{\R^{N}} |u|^{q-2}u\phi\, dx,
$$
$$
\int_{\R^{N}} f(u_{n})\phi \,dx\rightarrow \int_{\R^{N}} f(u)\phi \,dx,
$$
and using the fact that $\langle \J_{\mu}'(u_{n}),\phi \rangle=o_{n}(1)$, we can deduce that $\langle \J_{\mu}'(u),\phi \rangle=0$ which implies that $u$ is a critical point of $\J_{\mu}$.

Now, if $u\neq 0$, then $u$ is a nontrivial solution to \eqref{Pmu}.
Assume that $u=0$. Then $\|u_{n}\|_{\mu}\not \ri 0$ in $\X_{\mu}$. Hence, by Proposition \ref{prop1} there exist a sequence $\{y_{n}\}\subset \R^{N}$ and constants $R, \beta>0$ such that
\begin{align}\label{tv1}
\liminf_{n\ri \infty} \int_{\B_{R}(y_{n})} |u_{n}|^{q} dx \geq \beta>0.
\end{align}
Now, let us define
\begin{align*}
\tilde{v}_{n}(x):=u_{n}(x+y_{n}).
\end{align*}
From the invariance by translations of $\R^{N}$, it is clear that $\|\tilde{v}_{n}\|_{\mu}=\|u_{n}\|_{\mu}$, so $\{\tilde{v}_{n}\}$ is bounded in $\X_{\mu}$ and there exists $\tilde{v}$ such that $\tilde{v}_{n}\rightharpoonup \tilde{v}$ in $\X_{\mu}$, $\tilde{v}_{n}\ri \tilde{v}$ in $L^{m}_{loc}(\R^{N})$ for any $m\in [1, \q)$ and $\tilde{v}\neq 0$ in view of \eqref{tv1}. Moreover, $\J_{\mu}(\tilde{v}_{n})= \J_{\mu}(u_{n})$ and $\J'_{\mu}(\tilde{v}_{n})=o_{n}(1)$, and arguing as before it is easy to check that $\J_{\mu}'(\tilde{v})=0$.

Now let $u$ be the solution obtained from the previous study, and we prove that $u$ is a ground state solution. It is clear that $c_{\mu}\leq \J_{\mu}(u)$. On the other hand, from Fatou's Lemma we can see that
\begin{align*}
\J_{\mu}(u)= \J_{\mu}(u)- \frac{1}{q}\langle \J_{\mu}'(u), u\rangle \leq \liminf_{n\ri \infty} \left[\J_{\mu}(u_{n})- \frac{1}{q}\langle \J_{\mu}'(u_{n}), u_{n}\rangle\right]= c_{\mu},
\end{align*}
which implies that $c_{\mu}=\J_{\mu}(u)$.

Finally we prove that the ground state obtained before is nonnegative. Indeed, taking $u^{-}= \min \{u, 0\}$ as test function in \eqref{Pmu}, and exploiting $(f_{1})$ and the following inequality
\begin{align*}
|x- y|^{t-2} (x- y) (x^{-} - y^{-}) \geq |x^{-}- y^{-}|^{t} \quad \forall t\geq 1,
\end{align*}
we can see that
\begin{align*}
\|u^{-}\|_{\mu, p}^{p} + \|u^{-}\|_{\mu, q}^{q} \leq 0
\end{align*}
which implies that $u^{-}=0$, that is $u\geq 0$ in $\R^{N}$.
\end{proof}

Arguing as before, we can deduce the next result.
\begin{corollary}\label{cor3.2}
Let $\{u_{n}\}\subset \N_{\mu}$ be a sequence satisfying $\J_{\mu}(u_{n})\ri c_{\mu}$ with $u_{n}\rightharpoonup u$ in $\X_{\mu}$. If $u\neq 0$, then $u_{n}\ri u$ in $\X_{\mu}$ for some sequence. Otherwise, there exists a sequence $\{\tilde{y}_{n}\}\subset \R^{N}$ such that, up to a subsequence, $\tilde{v}_{n}(x)= u_{n}(x+ \tilde{y}_{n})$ strongly converges in $\X_{\mu}$.
\end{corollary}

\section{The non autonomous problem} \label{Sect4}
\noindent
In this section we deal with the following problem
\begin{equation}\label{Pe}
(-\Delta)_{p}^{s}u+ (-\Delta)_{q}^{s}u + V(\e x) (|u|^{p-2}u + |u|^{q-2}u)= f(u) \mbox{ in } \R^{N}.
\end{equation}
The corresponding functional is given by
\begin{align*}
\J_{\e}(u)= \frac{1}{p}[u]_{s, p}^{p} + \frac{1}{q}[u]_{s, q}^{q} + \int_{\R^{N}} V(\e x) \left( \frac{1}{p}|u|^{p}+\frac{1}{q}|u|^{q}\right) \,dx -\int_{\R^{N}} F(u) \, dx,
\end{align*}
which is well-defined on the space
\begin{align*}
\X_{\e}=\left \{u\in W^{s, p}(\R^{N})\cap W^{s, q}(\R^{N}) : \int_{\R^{N}} V(\e x) \left(|u|^{p}+|u|^{q}\right) \,dx<\infty \right \}
\end{align*}
endowed with the norm
\begin{align*}
\|u\|_{\e}= \|u\|_{V, p} + \|u\|_{V, q},
\end{align*}
where
\begin{align*}
\|u\|_{V, t}= \left([u]_{s, t}^{t} + \int_{\R^{N}} V(\e x) |u|^{t}\, dx\right)^{\frac{1}{t}} \quad \mbox{ for all } t>1.
\end{align*}
It is easy to check that $\J_{\e}\in \C^{1}(\X_{\e}, \R)$ and its differential is given by
\begin{align*}
\langle \J'_{\e}(u), \varphi \rangle &= \iint_{\R^{2N}} \frac{|u(x)-u(y)|^{p-2}(u(x)-u(y))(\varphi(x)- \varphi(y))}{|x-y|^{N+sp}}\, dxdy \\
&\quad + \iint_{\R^{2N}} \frac{|u(x)-u(y)|^{q-2}(u(x)-u(y))(\varphi(x)- \varphi(y))}{|x-y|^{N+sq}}\, dxdy \\
&\quad +\int_{\R^{N}}V(\e x) |u|^{p-2}u\,\varphi \, dx+ \int_{\R^{N}} V(\e x)|u|^{q-2}u\,\varphi \, dx   - \int_{\R^{N}} f(u)\varphi \, dx
\end{align*}
for any $u, \varphi \in \X_{\e}$.

Arguing as in Lemma \ref{lem1} we have that $\J_{\e}$ has a mountain pass geometry, and using the mountain pass theorem without $(PS)$ condition (see \cite{W})
there exists a $(PS)_{c_{\e}}$ sequence $\{u_{n}\}\subset \X_{\e}$ for $\J_{\e}$, where $c_{\e}$ is the minimax level associated to $\J_{\e}$. Moreover, proceeding as in Lemma \ref{lem3} we can see that $\{u_{n}\}$ is bounded in $\X_{\e}$. As in Lemma \ref{lem2} we can note that $c_{\e}$ has the following characterization
\begin{align*}
c_{\e}= \inf_{u\in \X_{\e}\setminus \{0\}} \sup_{t\geq 0} \J_{\e}(tu)= \inf_{u\in \N_{\e}} \J_{\e}(u),
\end{align*}
where $\N_{\e}=\{u\in \X_{\e}\setminus \{0\} : \langle \J'_{\e}(u), u\rangle =0\}$.

\begin{lemma}\label{lem4}
For all $u\in \N_{\e}$ there is a constant $\kappa>0$, independent of $\e$, such that
\begin{align*}
\|u\|_{\e}\geq \kappa.
\end{align*}
\end{lemma}

\begin{proof}
Since $f(t)\geq 0$ for all $t\geq 0$, we can see that
\begin{align}\label{tv2}
c_{\e}\leq \J_{\e}(u)&\leq \frac{1}{p}[u]_{s, p}^{p} + \frac{1}{q}[u]_{s, q}^{q} + \int_{\R^{N}} V(\e x) \left( \frac{1}{p}|u|_{p}^{p}+\frac{1}{q}|u|_{q}^{q}\right) \,dx\nonumber \\
&\leq C_{1} \|u\|_{V, p}^{p} + C_{2} \|u\|_{V, q}^{q} \leq C_{3} (\|u\|_{\e}^{p} + \|u\|_{\e}^{q}).
\end{align}
If $\|u\|_{\e}\geq 1$ for all $u\in \N_{\e}$, then the proof is done. Otherwise, if there exists $u\in \N_{\e}$ for which $\|u\|_{\e}\leq 1$, from \eqref{tv2} it follows that $c_{\e}\leq 2C_{3}\|u\|_{\e}^{p}$. Now, we observe that $\J_{V_{0}}(tu)\leq \J_{\e}(tu)$ for all $u\in \X_{\e}$, which gives $c_{V_{0}}\leq c_{\e}$. Therefore, we have $\left( \frac{c_{V_{0}}}{2C_{3}}\right)^{\frac{1}{p}} \leq \|u\|_{\e}$.
\end{proof}

Arguing as in the proof of Proposition \ref{prop1} it is possible to prove the following result.
\begin{proposition}\label{prop3}
Let $\{u_{n}\}\subset \N_{\e}$ be such that $\J_{\e}(u_{n})\ri c$ with $u_{n}\rightharpoonup 0$ in $\X_{\e}$. Then, one of the following alternatives occurs:
\begin{compactenum}[$(a)$]
\item $u_{n}\ri 0$ in $\X_{\e}$;
\item there are a sequence $\{y_{n}\}\subset \R^{N}$ and constants $R, \beta>0$ such that
\begin{align*}
\liminf_{n\ri \infty} \int_{\B_{R}(y_{n})} |u_{n}|^{q} dx \geq \beta>0.
\end{align*}
\end{compactenum}
\end{proposition}

Now we show the next lemma which will be useful to understand for which levels the functional $\J_{\e}$ verifies the Palais-Smale condition.
\begin{lemma}\label{lem5}
Assume that $V_{\infty}<\infty$ and let $\{v_{n}\}\subset \X_{\e}$ be a $(PS)_{c}$ sequence for $\J_{\e}$ with $v_{n}\rightharpoonup 0$ in $\X_{\e}$. If $v_{n}\not \ri 0$ in $\X_{\e}$, then $c\geq c_{V_{\infty}}$.
\end{lemma}

\begin{proof}
Let $\{t_{n}\}\subset (0, \infty)$ be such that $\{t_{n}v_{n}\}\subset \N_{V_{\infty}}$. \\
{\bf Claim 1:} Our aim is to show that $\limsup_{n\ri \infty} t_{n}\leq 1$. \\
Assume by contradiction that there exist $\delta>0$ and a subsequence, denoted again by $\{t_{n}\}$, such that
\begin{equation}\label{tv8}
t_{n}\geq 1+\delta \quad \mbox{ for any } n\in \mathbb{N}.
\end{equation}
Since $\{v_{n}\}\subset \X_{\e}$ is a bounded $(PS)$ sequence for $\J_{\e}$, we have that $\langle \J'_{\e}(v_{n}), v_{n}\rangle =o_{n}(1)$, or equivalently
\begin{align}\label{tv3}
[v_{n}]_{s, p}^{p} + [v_{n}]_{s, q}^{q} + \int_{\R^{N}} V(\e x) |v_{n}|^{p} dx + \int_{\R^{N}} V(\e x) |v_{n}|^{q} dx - \int_{\R^{N}} f(v_{n})v_{n}\, dx = o_{n}(1).
\end{align}
Since $t_{n}v_{n}\in \N_{V_{\infty}}$ we also have that
\begin{align}\label{tv4}
&t_{n}^{p-q}[v_{n}]_{s, p}^{p} + [v_{n}]_{s, q}^{q} + t_{n}^{p-q} V_{\infty} \int_{\R^{N}} |v_{n}|^{p} dx \nonumber\\
& + V_{\infty} \int_{\R^{N}} |v_{n}|^{q} dx - \int_{\R^{N}} \frac{f(t_{n}v_{n})}{(t_{n}v_{n})^{q-1}} |v_{n}|^{q}\, dx =0.
\end{align}
Putting together \eqref{tv3} and \eqref{tv4} we get
\begin{align}\label{tv5}
\int_{\R^{N}} \left(\frac{f(t_{n}v_{n})}{(t_{n}v_{n})^{q-1}}- \frac{f(v_{n})}{(v_{n})^{q-1}}\right) |v_{n}|^{q}\, dx \leq \int_{\R^{N}} (V_{\infty}-V(\e x)) |v_{n}|^{q} dx.
\end{align}
Now, using assumption \eqref{V0} we can see that, given $\zeta>0$ there exists $R=R(\zeta)>0$ such that
\begin{align}\label{hm}
V(\e x) \geq V_{\infty}-\zeta \quad \mbox{ for any } |x|\geq R.
\end{align}
From this, taking into account that $v_{n}\ri 0$ in $L^{q}(\B_{R})$ and the boundedness of $\{v_{n}\}$ in $\X_{\e}$, we can infer
\begin{align}\label{tv6}
\int_{\R^{N}} &(V_{\infty}-V(\e x)) |v_{n}|^{q} dx \nonumber \\
&= \int_{\B_{R}(0)}(V_{\infty}-V(\e x)) |v_{n}|^{q} dx + \int_{\R^{N}\setminus \B_{R}(0)}(V_{\infty}-V(\e x)) |v_{n}|^{q} dx  \nonumber \\
&\leq V_{\infty} \int_{\B_{R}(0)} |v_{n}|^{q} dx + \zeta \int_{\R^{N}\setminus \B_{R}(0)}|v_{n}|^{q} dx \nonumber \\
&\leq o_{n}(1) + \zeta C.
\end{align}
Combining \eqref{tv5} and \eqref{tv6} we have
\begin{equation}\label{tv9}
\int_{\R^{N}} \left(\frac{f(t_{n}v_{n})}{(t_{n}v_{n})^{q-1}}- \frac{f(v_{n})}{(v_{n})^{q-1}}\right) |v_{n}|^{q}\, dx \leq o_{n}(1) + \zeta C.
\end{equation}
Since $v_{n}\not\ri 0$ in $\X_{\e}$, we can apply Proposition \ref{prop3} to deduce the existence of a sequence $\{y_{n}\}\subset \R^{N}$ and two positive numbers $\bar{R}, \beta$ such that
\begin{align}\label{tv7}
\int_{\B_{\bar{R}}(y_{n})} |v_{n}|^{q}\, dx \geq \beta>0.
\end{align}
Let us consider $\tilde{v}_{n}=v_{n}(x+y_{n})$. Then we may assume that, up to a subsequence, $\tilde{v}_{n}\rightharpoonup \tilde{v}$ in $\X_{\e}$. By \eqref{tv7} there exists $\Omega\subset\R^{N}$ with positive measure and such that $\tilde{v}>0$ in $\Omega$. From \eqref{tv8}, $(f_{4})$ and \eqref{tv9}, we can infer that
\begin{equation*}
0<\int_{\Omega} \left(\frac{f((1+\delta)\tilde{v}_{n})}{((1+\delta)\tilde{v}_{n})^{q-1}}- \frac{f(\tilde{v}_{n})}{(\tilde{v}_{n})^{q-1}}\right) |\tilde{v}_{n}|^{q}\, dx \leq o_{n}(1) + \zeta C.
\end{equation*}
Taking the limit as $n\ri \infty$ and applying Fatou's Lemma, we obtain
\begin{equation*}
0<\int_{\Omega} \left(\frac{f((1+\delta)\tilde{v})}{((1+\delta)\tilde{v})^{q-1}}- \frac{f(\tilde{v})}{(\tilde{v})^{q-1}}\right) |\tilde{v}|^{q}\, dx \leq \zeta C \quad \mbox{ for any } \zeta>0,
\end{equation*}
and this is a contradiction.

Now, we will consider the following cases:

\textsc{Case 1:} Assume that $\limsup_{n\rightarrow \infty} t_{n}=1$. Thus, there exists $\{t_{n}\}$ such that $t_{n}\rightarrow 1$. Taking into account that  $\J_{\e}(v_{n})\rightarrow c$, we have
\begin{align}\label{tv12new}
c+ o_{n}(1)&= \J_{\e}(v_{n})\nonumber \\
&=\J_{\e}(v_{n}) - \J_{V_{\infty}}(t_{n}v_{n})+ \J_{V_{\infty}}(t_{n}v_{n}) \nonumber \\
&\geq \J_{\e}(v_{n}) -\J_{V_{\infty}}(t_{n}v_{n}) + c_{V_{\infty}}.
\end{align}
Let us compute $\J_{\e}(v_{n}) -\J_{V_{\infty}}(t_{n}v_{n})$:
\begin{align}\begin{split}\label{tv12}
&\J_{\e}(v_{n}) -\J_{V_{\infty}}(t_{n}v_{n}) \\
&\quad = \frac{1-t_{n}^{p}}{p} [v_{n}]^{p}_{s, p} + \frac{1-t_{n}^{q}}{q} [v_{n}]^{q}_{s, q} + \frac{1}{p} \int_{\R^{N}} \left( V(\e x) - t_{n}^{p} V_{\infty}\right) |v_{n}|^{p} dx  \\
&\qquad +\frac{1}{q} \int_{\R^{N}} \left( V(\e x) - t_{n}^{q} V_{\infty}\right) |v_{n}|^{q} dx + \int_{\R^{N}} \left( F(t_{n} v_{n}) -F(v_{n}) \right) \, dx.
\end{split} \end{align}
Using condition \eqref{V0},  $v_{n}\rightarrow 0$ in $L^{p}(\B_{R}(0))$, $t_{n}\rightarrow 1$, \eqref{hm}, and
\begin{align*}
V(\e x) \!-\! t_{n}^{p} V_{\infty} =\left(V(\e x) \!-\! V_{\infty} \right)\! +\! (1\!-\! t_{n}^{p}) V_{\infty}\geq \!-\!\zeta \!+\! (1\!-\! t_{n}^{p}) V_{\infty} \ \ \mbox{ for any } |x|\geq R,
\end{align*}
we get
\begin{align}\label{tv13}
&\int_{\R^{N}}  \left( V(\e x) - t_{n}^{p} V_{\infty}\right) |v_{n}|^{p} dx \nonumber \\
&= \int_{\B_{R}(0)} \left( V(\e x) - t_{n}^{p} V_{\infty}\right) |v_{n}|^{p} dx+ \int_{\R^{N}\setminus \B_{R}(0)} \left( V(\e x) - t_{n}^{p} V_{\infty}\right) |v_{n}|^{p} dx \nonumber \\
&\geq (V_{0}\!-\! t_{n}^{p}V_{\infty}) \!\int_{\B_{R}(0)}\! |v_{n}|^{p} dx \!-\! \zeta \int_{\R^{N}\setminus \B_{R}(0)} |v_{n}|^{p} dx\!+\! V_{\infty}(1\!-\!t_{n}^{p}) \int_{\R^{N}\setminus \B_{R}(0)} |v_{n}|^{p} dx \nonumber \\
&\geq o_{n}(1)- \zeta C.
\end{align}
In similar fashion we can prove that
\begin{align}\label{tv131}
\int_{\R^{N}} & \left( V(\e x) - t_{n}^{q} V_{\infty}\right) |v_{n}|^{q} dx \geq o_{n}(1)- \zeta C.
\end{align}
Since $\{v_{n}\}$ is bounded in $\X_{\e}$, we can conclude that
\begin{align}\label{tv14}
\frac{(1-t_{n}^{p})}{p} [v_{n}]^{p}_{s, p}= o_{n}(1) \quad \mbox{ and }\quad \frac{(1-t_{n}^{q})}{q} [v_{n}]^{q}_{s, q}= o_{n}(1).
\end{align}
Thus, putting together \eqref{tv12}, \eqref{tv13}, \eqref{tv131} and \eqref{tv14} we obtain
\begin{align}\label{tv15}
\J_{\e}(v_{n}) -\J_{V_{\infty}}(t_{n}v_{n}) \geq \int_{\R^{N}} \left( F(t_{n} v_{n}) -F(v_{n}) \right) \, dx +o_{n}(1)- \zeta C.
\end{align}
At this point, our aim is to show that
\begin{align}\label{tv16}
\int_{\R^{N}} \left( F(t_{n} v_{n}) -F(v_{n}) \right) \, dx=o_{n}(1).
\end{align}
Applying the Mean Value Theorem and \eqref{growthf} we deduce that
\begin{align*}
\int_{\R^{N}} | F(t_{n} v_{n}) -F(v_{n}) | \, dx \leq C|t_{n}-1|\int_{\R^{N}} |v_{n}|^{p} dx + C|t_{n}-1|\int_{\R^{N}} |v_{n}|^{r} dx.
\end{align*}
Exploiting the boundedness of $\{v_{n}\}$ we get the thesis. Now, gathering \eqref{tv12new}, \eqref{tv15} and \eqref{tv16} we can infer that
\begin{align*}
c+ o_{n}(1)\geq o_{n}(1) - \zeta C + c_{V_{\infty}},
\end{align*}
and, taking the limit as $\zeta \ri 0$ we get $c \geq c_{V_{\infty}}$.

\textsc{Case 2:} Assume that $\limsup_{n\rightarrow \infty} t_{n}=t_{0}<1$. Then, there is a subsequence, still denoted by $\{t_{n}\}$, such that $t_{n}\rightarrow t_{0} (<1)$ and $t_{n}<1$ for any $n\in \mathbb{N}$.
Let us observe that
\begin{align}\label{tv17}
c+o_{n}(1)&= \J_{\e}(v_{n}) - \frac{1}{q}\langle \J'_{\e}(v_{n}), v_{n} \rangle \nonumber \\
&= \left(\frac{1}{p}- \frac{1}{q}\right) \|v_{n}\|_{V, p}^{p}  + \int_{\R^{N}} \left(\frac{1}{q}f(v_{n}) v_{n} - F(v_{n})\right) \,dx.
\end{align}
Exploiting the fact that $t_{n}v_{n}\in \N_{V_{\infty}}$, and using $(f_{5})$ and \eqref{tv17}, we obtain
\begin{align*}
c_{V_{\infty}} &\leq \J_{V_{\infty}}(t_{n}v_{n})  \\
&= \J_{V_{\infty}}(t_{n}v_{n}) - \frac{1}{q} \langle \J'_{V_{\infty}}(t_{n}v_{n}), t_{n}v_{n} \rangle \\
&=\left(\frac{1}{p}- \frac{1}{q}\right) \|t_{n}v_{n}\|_{V, p}^{p} +  \int_{\R^{N}} \left(\frac{1}{q} f(t_{n}v_{n}) t_{n}v_{n}- F(t_{n}v_{n})\right) \, dx \\
&\leq \left(\frac{1}{p}- \frac{1}{q}\right) \|v_{n}\|_{V, p}^{p} + \int_{\R^{N}} \left(\frac{1}{q}f(v_{n}) v_{n} - F(v_{n})\right) \,dx \\
&=c +o_{n}(1).
\end{align*}
Taking the limit as $n\rightarrow \infty$ we get $c\geq c_{V_{\infty}}$.
\end{proof}

Now we show the following compactness result.
\begin{proposition}\label{prop4}
The functional $\J_{\e}$ restricted to $\N_{\e}$ satisfies the $(PS)_{c}$ condition at any level $c\in (0, c_{V_{\infty}})$ if $V_{\infty}<\infty$ and at any level $c\in \R$ if $V_{\infty}= \infty$.
\end{proposition}

\begin{proof}
Let $\{u_{n}\}\subset \N_{\e}$ be such that $\J_{\e}(u_{n})\ri c$ and $\|\J'_{\e}(u_{n})\|_{*}=o_{n}(1)$, where $\|\J'_{\e}(u_{n})\|_{*}$ denotes the norm of the derivative of the restriction of $\J_{\e}$ to $\N_{\e}$ at $u$. Then, there exists $\{\la_{n}\}\subset \R$ such that
\begin{align}\label{nonumb2}
\J'_{\e}(u_{n})= \la_{n} T'_{\e}(u_{n})+ o_{n}(1),
\end{align}
where $T_{\e}: \X_{\e}\ri \R$ is defined as
\begin{align*}
T_{\e}(u)= \|u\|_{V, p}^{p}+ \|u\|_{V, q}^{q}- \int_{\R^{N}} f(u)u\, dx.
\end{align*}
From $\{u_{n}\}\subset \N_{\e}$ and $(f_{5})$ it follows that
\begin{align}\label{nonumb1}
\langle T'_{\e}(u_{n}), u_{n} \rangle &= [u_{n}]_{s, p}^{p}+ [u_{n}]_{s, q}^{q} +\int_{\R^{N}} V(\e x) (p|u_{n}|^{p}+q|u_{n}|^{q}) \, dx \nonumber \\
&\quad  - \int_{\R^{N}} f(u_{n})u_{n}\, dx - \int_{\R^{N}} f'(u_{n})|u_{n}|^{2}\, dx \nonumber \\
&\leq - \int_{\R^{N}} \left( f'(u_{n})|u_{n}|^{2}-(q-1) f(u_{n})u_{n} \right)\, dx.
\end{align}
Since $\{u_{n}\}$ is bounded in $\X_{\e}$ and $\|u_{n}\|_{\e}\not\ri 0$, by Proposition \ref{prop3} there exists a sequence $\{y_{n}\}\subset \R^{N}$ such that $\tilde{u}_{n}= u_{n}(\cdot +y_{n})$ is bounded in $\X_{\e}$ and $\tilde{u}_{n}\rightharpoonup \tilde{u}$ in $\X_{\e}$ for some $\tilde{u}\neq 0$. Therefore, there exists a set $\Omega \subset \R^{N}$ with positive measure such that $\tilde{u}>0$ in $\Omega$.

Suppose by contradiction that
$$
\limsup_{n\ri \infty} \, \langle T'_{\e}(u_{n}), u_{n}\rangle=0.
$$
Then, in the light of \eqref{nonumb1}, $(f_{5})$ and Fatou's Lemma we have
\begin{align*}
0\leq - \int_{\Omega} \left( f'(\tilde{u})|\tilde{u}|^{2}-(q-1) f(\tilde{u})\tilde{u} \right)\, dx < 0.
\end{align*}
which gives a contradiction. Thus $\limsup_{n\ri \infty} \langle T'_{\e}(u_{n}), u_{n}\rangle<0$ and, as a consequence, $\la_{n}=o_{n}(1)$. This and \eqref{nonumb2}, implies that $\{u_{n}\}$ is a $(PS)_{c}$ sequence.

Let $v_{n}=u_{n}-u$. Then, by Lemma \ref{lem7} we can infer that
\begin{align*}
\J_{\e}(v_{n})= \J_{\e}(u_{n}) - \J_{\e}(u) + o_{n}(1)= c- \J_{\e}(u) + o_{n}(1)=: c_{*}+ o_{n}(1)
\end{align*}
and
$$
\J'_{\e}(v_{n})=o_{n}(1).
$$
Now, using $(f_{5})$ we have
\begin{align*}
\J_{\e}(u)&= \J_{\e}(u)- \frac{1}{q}\langle \J'_{\e}(u), u\rangle \\
&=\left(\frac{1}{p}-\frac{1}{q}\right) [u]_{s, p}^{p} + \left(\frac{1}{p}-\frac{1}{q}\right)\int_{\R^{N}} V(\e x) |u|^{p} dx - \int_{\R^{N}} \left( F(u)-\frac{1}{q}f(u)u\right) \, dx\\
&=\left(\frac{1}{p}-\frac{1}{q}\right) \|u\|_{V, p}^{p} - \int_{\R^{N}} \left( F(u)-\frac{1}{q}f(u)u\right) \, dx\\
&\geq \left(\frac{1}{p}-\frac{1}{q}\right) \|u\|_{V, p}^{p} \geq 0.
\end{align*}

If we assume that $V_{\infty}<\infty$, then $c_{*}\leq c <c_{V_{\infty}}$, and applying Lemma \ref{lem5} we can infer $v_{n}\ri 0$ in $\X_{\e}$, that is $u_{n}\ri u$ in $\X_{\e}$.

Now we consider the case $V_{\infty}=\infty$. By Lemma \ref{lem6} we have that $\X_{\e}$ is compactly embedded in $L^{m}(\R^{N})$ for any $p\leq m<\q$. Up to a subsequence, we have that $v_{n}\ri 0$ in $L^{t}(\R^{N})$ for all $t\in [q, \q)$. From \eqref{growthf} we get
\begin{align*}
\int_{\R^{N}} f(v_{n})v_{n} \, dx=o_{n}(1).
\end{align*}
Therefore $\|v_{n}\|_{V, p}^{p}+ \|v_{n}\|_{V, q}^{q}= o_{n}(1)$, which implies that $u_{n}\ri u$ in $\X_{\e}$.
\end{proof}

As a byproduct of the above proof, we have the following result.
\begin{corollary}\label{cor12}
The critical points of $\J_{\e}$ restricted to $\N_{\e}$ are critical points of $\J_{\e}$ on $\X_{\e}$.
\end{corollary}

The forthcoming result regards the existence of a nonnegative ground state solution to \eqref{Pe} provided that $\e>0$ is small enough.

\begin{theorem}\label{thm41}
Assume that \eqref{V0}, $(f_{1})$-$(f_{5})$ hold. Then, there exists $\bar{\e}>0$ such that problem \eqref{Pe} has a nonnegative ground state solution $u_{\e}$ for all $\e \in (0, \bar{\e})$.
\end{theorem}

\begin{proof}
It is easy to see that $\J_{\e}$ has a mountain pass geometry. Thus, there exists a bounded sequence $\{u_{n}\}\subset \X_{\e}$ such that
\begin{align*}
\J_{\e}(u_{n})\ri c_{\e} \quad \mbox{ and }\quad \J'_{\e}(u_{n})\ri 0.
\end{align*}
Let us consider the case $V_{\infty}=\infty$. Then from Lemma \ref{lem6}, up to a subsequence, we get $u_{n}\ri u$ in $L^{m}(\R^{N})$ for all $m\in [p, \q)$, for some $u\in \X_{\e}$. Set $v_{n}=u_{n}-u$. From assumptions $(f_{2})$ and $(f_{3})$, and the Dominated Convergence Theorem we can infer
\begin{align*}
\int_{\R^{N}} f(v_{n})v_{n}\, dx=o_{n}(1).
\end{align*}
Then $\|v_{n}\|_{\e}=o_{n}(1)$, that is $u_{n}\ri u$ in $\X_{\e}$ and we can deduce that $\J_{\e}(u)=c_{\e}$ and $\J'_{\e}(u)=0$.

Now, if $V_{\infty}<\infty$, assume without loss of generality that
\begin{align*}
V(0)=V_{0}= \inf_{x\in \R^{N}} V(x).
\end{align*}
Let $\mu \in (V_{0}, V_{\infty})$ and we note that $c_{V_{0}}<c_{\mu}<c_{V_{\infty}}$. Let us prove that there exists a function $w\in \X_{\mu}$ with compact support such that
\begin{align}\label{D1}
\J_{\mu}(w) =\max_{t\geq 0} \J_{\mu}(tw) \quad \mbox{ and }\quad \J_{\mu}(w)<c_{V_{\infty}}.
\end{align}
Let $\psi\in \C^{\infty}_{c}(\R^{N}, [0, 1])$ be such that $\psi=1$ in $\B_{1}(0)$ and $\psi=2$ in $\R^{N}\setminus \B_{2}(0)$. For any $R>0$, we set $\psi_{R}(x)= \psi(\frac{x}{R})$ and we consider the function $w_{R}(x)= \psi_{R}(x)w^{\mu}(x)$, where $w^{\mu}$ is a ground state solution to \eqref{Pmu}. By Lemma \ref{Psi} we can see that
\begin{align}\label{FS1}
\lim_{R\ri \infty} \|w_{R}-w^{\mu}\|_{s, p} + \|w_{R}-w^{\mu}\|_{s, q} =0.
\end{align}
Let $t_{R}>0$ be such that $\J_{\mu}(t_{R}w_{R})= \max_{t\geq 0} \J_{\mu}(tw_{R})$. Then, $t_{R}w_{R}\in \N_{\mu}$.
Now we can see that there exists $\bar{r}>0$ such that $\J_{\mu}(t_{\bar{r}}w_{\bar{r}})<c_{V_{\infty}}$. Indeed, if $\J_{\mu}(t_{R}w_{R}) \geq c_{V_{\infty}}$ for any $R>0$, using $t_{R}w_{R}\in \N_{\mu}$, \eqref{FS1} and $w^{\mu}$ is a ground state, we can deduce that $t_{R}\ri 1$ and
\begin{align*}
c_{V_{\infty}}\leq \liminf_{R\ri \infty} \J_{\mu}(t_{R}w_{R}) = \J_{\mu}(w^{\mu})=c_{\mu},
\end{align*}
which gives a contradiction. Then, taking $w=\psi_{\bar{r}}w^{\mu}$, we can conclude that \eqref{D1} holds true.

Now, from \eqref{V0} it follows that for some $\bar{\e}>0$
\begin{align}\label{D2}
V(\e x)\leq \mu \quad \mbox{ for all } x\in \supp w \mbox{ and } \e \leq \bar{\e}.
\end{align}
Then, in the light of \eqref{D1} and \eqref{D2} we have for all  $\e \leq \bar{\e}$
\begin{align*}
\J_{\e}(tw)\leq \J_{\mu}(tw) \leq \J_{\mu}(w) \quad \mbox{ for all } t>0,
\end{align*}
which implies that
\begin{align*}
\max_{t> 0} \J_{\e}(tw) \leq \J_{\mu}(w).
\end{align*}
This proves that, for $\e$ small enough, $c_{\e}<c_{V_{\infty}}$. We can conclude the proof in view of Proposition \ref{prop4}.
\end{proof}

In what follows we show an interesting relation between $c_{\e}$ and $c_{V_{0}}$.
\begin{lemma}\label{ce}
$\lim_{\e\ri 0} c_{\e}= c_{V_{0}}$.
\end{lemma}

\begin{proof}
From \eqref{V0}, it is clear that $\displaystyle{\liminf_{\e\ri 0} c_{\e} \geq c_{V_{0}}}$. Now we prove that
$$\limsup_{\e\ri 0} c_{\e} \leq c_{V_{0}}.$$

Let $\psi\in \C^{\infty}_{c}(\R^{N}, [0, 1])$ be such that $\psi=1$ in $\B_{1}(0)$ and $\psi=2$ in $\R^{N}\setminus \B_{2}(0)$. For any $R>0$, we set $\psi_{R}(x)= \psi(\frac{x}{R})$ and we consider the function $w_{R}(x)= \psi_{R}(x)w(x)$, where $w$ is a ground state solution to \eqref{Pmu} with $\mu=V_{0}$. By Lemma \ref{Psi} we can see that
\begin{align}\label{FS11}
\lim_{R\ri \infty} \|w_{R}-w\|_{s, p} + \|w_{R}-w\|_{s, q} =0.
\end{align}
For any $\e, R>0$ let $t_{\e, R}>0$ be such that
$$
\J_{\e}(t_{\e, R}w_{R})= \max_{t\geq 0} \J_{\e}(tw_{R}).
$$
Then,
\begin{align}\label{FS12}
\langle \J_{\e}'(t_{\e,R}w_{R}), t_{\e,R}w_{R} \rangle =0
\end{align}
from which we can deduce that for any $R>0$ it holds
\begin{align}\label{FS13}
0<\lim_{\e\ri 0} t_{\e, R} = t_{R}<\infty.
\end{align}
Indeed,
\begin{align*}
\|w_{R}\|_{V, p}^{p} + t_{\e, R}^{q-p} \,\|w_{R}\|_{V, q}^{q}= \int_{\R^{N}} \frac{f(t_{\e, R} w_{R})}{(t_{\e, R} w_{R})^{p-1}} \,w_{R}^{p}\, dx
\end{align*}
and using $p<q$ and $(f_{2})$ we can see that \eqref{FS13} holds true.

Now, from \eqref{FS12} we know that
\begin{align*}
t_{\e, R}^{p-q}\, \|w_{R}\|_{V, p}^{p} + \|w_{R}\|_{V, q}^{q}= \int_{\R^{N}} \frac{f(t_{\e, R} w_{R})}{(t_{\e, R} w_{R})^{q-1}} \,w_{R}^{q}\, dx
\end{align*}
and taking the  limit as $\e\ri 0$ and exploiting \eqref{FS13}, we find
\begin{align*}
t_{R}^{p-q}\, \|w_{R}\|_{V, p}^{p} + \|w_{R}\|_{V, q}^{q}= \int_{\R^{N}} \frac{f(t_{R} w_{R})}{(t_{R} w_{R})^{q-1}} \,w_{R}^{q}\, dx.
\end{align*}
Letting $R\ri \infty$ and from \eqref{FS11}, $w\in \N_{V_{0}}$ and $(f_{5})$
we can deduce that
\begin{equation}\label{FS14}
\lim_{R\ri \infty} t_{R}=1 \quad \mbox{ and } \quad \J_{V_{0}}(t_{R}w_{R}) = \max_{t\geq 0} \J_{V_{0}} (tw_{R}).
\end{equation}
Consequently we have
\begin{align*}
c_{\e}\leq \max_{t\geq 0} \J_{\e}(tw_{R})= \J_{\e}(t_{\e, R}w_{R}),
\end{align*}
which together with \eqref{FS13} implies that $\limsup_{\e\ri 0} c_{\e} \leq \J_{V_{0}}(t_{R}w_{R})$. Passing to the limit as $R\ri \infty$ and using \eqref{FS11} and \eqref{FS14} we can see that $\limsup_{\e\ri 0} c_{\e} \leq c_{V_{0}}$.
\end{proof}

The next result concerns the concentration phenomenon of solutions to \eqref{P} around the minima points of $V$.

\begin{proposition}\label{prop61}
Let $\e_{n}\ri 0$ and $\{u_{\e_{n}}\}\subset \X_{\e_{n}}$ be such that $\J_{\e_{n}}(u_{\e_{n}})= c_{\e_{n}}$ and $\J_{\e_{n}}'(u_{\e_{n}})=0$. Then, there exists a sequence $\{\tilde{y}_{n}\}\subset \R^{N}$ such that $\tilde{v}_{n}(x)=u_{\e_{n}}(x+ \tilde{y}_{n})$ has a convergent subsequence in $\X_{V_{0}}$. Moreover, up to a subsequence, $y_{n}\ri y$ for some $y\in \R^{N}$ such that $V(y)= V_{0}$, where $y_{n}=\e_{n}\tilde{y}_{n}$, and there exists $C>0$ such that for all $\delta>0$, there exist $\bar{R}>0$ and $n_{0}\in \mathbb{N}$ such that
\begin{align*}
\int_{\R^{N}\setminus \B_{\e_{n}\bar{R}}(y)} f(z_{n})z_{n}\,dx<\e_{n}^{N}\delta \quad \mbox{ and } \quad \int_{\B_{\e_{n}\bar{R}}(y)} f(z_{n})z_{n}\,dx\geq C\e_{n}^{N}
\end{align*}
for all $n\geq n_{0}$, where $z_{n}=u_{n}(\frac{\cdot}{\e_{n}})$.
\end{proposition}

\begin{proof}
For simplicity, we set $u_{n}:=u_{\e_{n}}$. In view of Lemma \ref{ce} and arguing as in Proposition \ref{prop1} we can find a sequence $\{\tilde{y}_{n}\}$ and constants $R, \beta>0$ such that
\begin{equation*}
\liminf_{n\ri \infty} \int_{\B_{R}(\tilde{y}_{n})} |u_{n}|^{q} dx \geq \beta >0.
\end{equation*}
Since $\J_{\e_{n}}(u_{n})= c_{\e_{n}}$ and $\J_{\e_{n}}'(u_{n})=0$, and by Lemma \ref{ce}, we deduce that $\{u_{n}\}$ is bounded in $\X_{V_{0}}$. Set $\tilde{v}_{n}(x)=u_{n}(x+ \tilde{y}_{n})$. Then, $\{\tilde{v}_{n}\}$ is bounded in $\X_{V_{0}}$ and we may assume that $\tilde{v}_{n}\rightharpoonup \tilde{v}$ in $\X_{V_{0}}$ for some $\tilde{v}\not \equiv 0$. Now, let $t_{n}>0$ be such that $v_{n}= t_{n}\tilde{v}_{n}\in \N_{V_{0}}$ and we define $y_{n}= \e_{n} \tilde{y}_{n}$.
Therefore
\begin{align*}
c_{V_{0}} &\leq \J_{V_{0}}(v_{n}) \\
&\leq \frac{1}{p}[v_{n}]_{s, p}^{p} \!+\! \frac{1}{q} [v_{n}]_{s, q}^{q}\! +\! \int_{\R^{N}} V(\e_{n} x \!+\! y_{n}) \left(\frac{1}{p}|v_{n}|^{p} \!+\!\frac{1}{q} |v_{n}|^{q}\right) dx\! -\!\int_{\R^{N}} \! F(v_{n})dx\\
&\leq \frac{t_{n}^{p}}{p}[u_{n}]_{s, p}^{p} \!+\! \frac{t_{n}^{q}}{q} [u_{n}]_{s, q}^{q} \!+\! \int_{\R^{N}}\! V(\e_{n} x) \left(\frac{t_{n}^{p}}{p}|u_{n}|^{p} \!+\!\frac{t_{n}^{q}}{q} |u_{n}|^{q}\right) dx \!-\!\int_{\R^{N}} \!F(t_{n} u_{n}) dx\\
&\leq \max_{t\geq 0} \J_{\e_{n}} (tu_{n}) = \J_{\e_{n}} (u_{n})= c_{V_{0}}+ o_{n}(1),
\end{align*}
which implies
\begin{equation}\label{AK5.31}
\J_{V_{0}}(v_{n})\ri c_{V_{0}} \quad \mbox{ and } \quad \{v_{n}\}\subset \N_{V_{0}}.
\end{equation}
Moreover, $\{v_{n}\}$ is bounded, and we may assume that $v_{n}\rightharpoonup v$ in $\X_{V_{0}}$. Since $\|v_{n}\|_{V_{0}}\not \ri 0$ and $\{v_{n}\}$ is bounded, there exists $t_{0}>0$ such that $t_{n}\ri t_{0}$. In the light of \eqref{AK5.31} and Corollary \ref{cor3.2} we can see that $v_{n}\ri t_{0}\tilde{v}= v$ in $\X_{V_{0}}$, which yields $\tilde{v}_{n}\ri \tilde{v}$ in $\X_{V_{0}}$.

In what follows, we show that $y_{n}\ri y$ for some $y\in \R^{N}$ such that $V(y)= V_{0}$. Firstly, we prove that $\{y_{n}\}$ is bounded in $\R^{N}$. We argue by contradiction, and assume, up to subsequence, that $|y_{n}|\ri +\infty$.

Consider the case $V_{\infty}=\infty$. Since $\langle \J_{\e_{n}}'(u_{n}), u_{n}\rangle =0$, we can see that
\begin{align*}
\int_{\R^{N}} &V(\e_{n} x + y_{n}) \left(|\tilde{v}_{n}|^{p} +|\tilde{v}_{n}|^{q}\right) \, dx \\
&\leq [\tilde{v}_{n}]_{s, p}^{p}+ [\tilde{v}_{n}]_{s, q}^{q} + \int_{\R^{N}} V(\e_{n} x + y_{n}) \left(|\tilde{v}_{n}|^{p} +|\tilde{v}_{n}|^{q}\right) \, dx = \int_{\R^{N}} f(\tilde{v}_{n})\tilde{v}_{n}\, dx.
\end{align*}
Taking into account \eqref{V0}, $\tilde{v}_{n}\ri \tilde{v}$ in $\X_{V_{0}}$, $(f_{2})$-$(f_{3})$ and using Fatou's Lemma, we obtain
\begin{align*}
+\infty &= \liminf_{n\ri \infty} \int_{\R^{N}} V(\e_{n} x + y_{n}) \left(|\tilde{v}_{n}|^{p} +|\tilde{v}_{n}|^{q}\right) \, dx \\
&\leq \liminf_{n\ri \infty}  \int_{\R^{N}} f(\tilde{v}_{n})\tilde{v}_{n}\, dx = \int_{\R^{N}} f(\tilde{v}) \tilde{v}\, dx <\infty,
\end{align*}
which gives a contradiction.

Now, we suppose $V_{\infty}<\infty$. Then, by Fatou's Lemma, \eqref{V0}, $\tilde{v}_{n}\ri \tilde{v}$ in $\X_{V_{0}}$ and a change of variable, we can deduce that
\begin{align*}
c_{V_{0}} \leq& \J_{V_{0}}(v) < \J_{V_{\infty}}(v) \\
\leq &\liminf_{n\ri \infty} \left[\frac{1}{p}[v_{n}]_{s, p}^{p} + \frac{1}{q} [v_{n}]_{s, q}^{q} \right.\\
&\left.+ \int_{\R^{N}} V(\e_{n} x + y_{n}) \left(\frac{1}{p}|v_{n}|^{p} +\frac{1}{q} |v_{n}|^{q}\right) \, dx -\int_{\R^{N}} F(v_{n}) \, dx\right]\\
=&\liminf_{n\ri \infty} \left[ \frac{t_{n}^{p}}{p}[u_{n}]_{s, p}^{p} + \frac{t_{n}^{q}}{q} [u_{n}]_{s, q}^{q} \right.\\
&\left.+ \int_{\R^{N}} V(\e_{n} x) \left(\frac{t_{n}^{p}}{p}|u_{n}|^{p} +\frac{t_{n}^{q}}{q} |u_{n}|^{q}\right) \, dx -\int_{\R^{N}} F(t_{n} u_{n}) \, dx\right]\\
=& \liminf_{n\ri \infty}  \J_{\e_{n}}(t_{n}u_{n}) \leq \liminf_{n\ri \infty}  \J_{\e_{n}}(u_{n})= c_{V_{0}},
\end{align*}
which is impossible. Hence, we can find $y\in \R^{N}$ such that $y_{n}\ri y$. Arguing as before, it is easy to check that $V(y)=V_{0}$. Now, since $v_{n}\ri v$ in $\X_{V_{0}}$ we have that
\begin{align*}
\int_{\R^{N}} f(v_{n})v_{n} \, dx \ri \int_{\R^{N}} f(v)v\, dx.
\end{align*}
Set $L=\int_{\R^{N}} f(v)vdx>0$. Then, for a given $\delta>0$ there exist $R>0$ and $n_{0}\in \mathbb{N}$ such that for all $n\geq n_{0}$
$$
\int_{\R^{N}\setminus \B_{R}(0)} f(v_{n})v_{n}dx<\delta
$$
which implies that
$$
\int_{\B_{R}(0)} f(v_{n})v_{n}dx\geq L-\delta+o_{n}(1).
$$
Since $\tilde{v}_{n}(x)=u_{n}(x+\tilde{y}_{n})=z_{n}(\e_{n}x +y_{n})$ we can see that
$$
\int_{\R^{N}\setminus \B_{\e_{n}R}(y_{n})} f(z_{n})z_{n}dx<\e_{n}^{N}\delta
$$
and for some $C>0$
$$
\int_{\B_{\e_{n}R}(y_{n})} f(z_{n})z_{n}dx\geq C\e_{n}^{N}
$$
for $n\geq n_{0}$. On the other hand, $y_{n}\ri y$ with $V(y)=V_{0}$, so we can find $\bar{R}>0$ such that $\B_{R}(y_{n})\subset \B_{\bar{R}}(y)$ for $n\geq n_{0}$. Consequently, for all $n\geq n_{0}$ it holds
$$
\int_{\R^{N}\setminus \B_{\e_{n}\bar{R}}(y)} f(z_{n})z_{n}dx<\e_{n}^{N}\delta
$$
and
$$
\int_{\B_{\e_{n}\bar{R}}(y)} f(z_{n})z_{n}dx\geq C\e_{n}^{N}.
$$
\end{proof}

\section{Multiplicity of solutions to \eqref{P}}\label{Sect5}
\noindent
This section is devoted to the study of the multiplicity of solutions to \eqref{P}.
The next result will be fundamental to implement the barycenter machinery. Since the proof is similar to the one given in Proposition \ref{prop61} we skip the details.

\begin{proposition}\label{prop6}
Let $\e_{n}\ri 0$ and $\{u_{n}\}\subset \N_{\e_{n}}$ be such that $\J_{\e_{n}}(u_{n})\ri c_{V_{0}}$. Then, there exists a sequence $\{\tilde{y}_{n}\}\subset \R^{N}$ such that $\tilde{v}_{n}(x)=u_{n}(x+ \tilde{y}_{n})$ has a convergent subsequence in $\X_{V_{0}}$. Moreover, up to a subsequence, $y_{n}\ri y\in M$, where $y_{n}=\e_{n}\tilde{y}_{n}$.
\end{proposition}

Now, fix $\delta>0$ and let $\omega$ be a ground state solution of autonomous problem \eqref{Pmu} with $\mu=V_{0}$, that is $\J_{V_{0}}(\omega)= c_{V_{0}}$ and $\J'_{V_{0}}(\omega)=0$.
Let $\eta\in \C^{\infty}([0, \infty), [0, 1])$ be a nonincreasing cut-off function such that $\eta(\tau)=1$ if $\tau\in [0, \frac{\delta}{2}]$ and $\eta(\tau)=0$ if $\tau \geq \delta$.
For any $y\in M$, we define
\begin{equation*}
\Upsilon_{\e, y}(x)= \eta(|\e x-y|) \omega \left(\frac{\e x-y}{\e}\right).
\end{equation*}
Let $t_{\e}>0$ be the unique number such that
\begin{align*}
\J_{\e}(t_{\e}\Upsilon_{\e, y})= \max_{t\geq 0} \J_{\e}(t \Upsilon_{\e, y})
\end{align*}
and let us introduce the map $\Phi_{\e}:M\ri \N_{\e}$ defined as $\Phi_{\e}(y):=t_{\e}\Upsilon_{\e, y}$. By construction, $\Phi_{\e}(y)$ has compact support for any $y\in M$. Then, we can prove that

\begin{lemma}\label{lem9}
The functional $\Phi_{\e}$ satisfies the following limit
\begin{equation}\label{3.2}
\lim_{\e\rightarrow 0} \J_{\e}(\Phi_{\e}(y))=c_{V_{0}} \mbox{ uniformly in } y\in M.
\end{equation}
\end{lemma}
\begin{proof}
Assume by contradiction that there there exist $\delta_{0}>0$, $\{y_{n}\}\subset M$ and $\e_{n}\rightarrow 0$ such that
\begin{equation}\label{4.41}
|\J_{\e_{n}}(\Phi_{\e_{n}} (y_{n}))-c_{V_{0}}|\geq \delta_{0}.
\end{equation}
Let us observe that Lemma \ref{Psi} and the Dominated Convergence Theorem imply
\begin{align}\label{4.421}
[\Upsilon_{\e_{n}, y_{n}}]_{s, p}^{p} + \int_{\R^{N}} V(\e_{n}x) |\Upsilon_{\e_{n}, y_{n}}|^{p} \, dx \ri [\omega]_{s, p}^{p} +\int_{\R^{N}} V_{0}|\omega|^{p}\, dx
\end{align}
and
\begin{align}\label{4.422}
[\Upsilon_{\e_{n}, y_{n}}]_{s, q}^{q} + \int_{\R^{N}} V(\e_{n}x) |\Upsilon_{\e_{n}, y_{n}}|^{q} \, dx \ri [\omega]_{s, q}^{q} +\int_{\R^{N}} V_{0}|\omega|^{q}\, dx.
\end{align}
Since $\langle \J'_{\e_{n}}(t_{\e_{n}}\Upsilon_{\e_{n}, y_{n}}), t_{\e_{n}} \Upsilon_{\e_{n}, y_{n}}\rangle=0$,
we can use the change of variable ${z=\frac{\e_{n}x-y_{n}}{\e_{n}}}$ to see that
\begin{align}\label{4.411}
&[t_{\e_{n}} \Upsilon_{\e_{n}, y_{n}}]_{s, p}^{p}+ [t_{\e_{n}}\Upsilon_{\e_{n}, y_{n}}]_{s, q}^{q} + \int_{\R^{N}} V(\e_{n}x) \left(  |t_{\e_{n}}\Upsilon_{\e_{n}, y_{n}}|^{p} +|t_{\e_{n}}\Upsilon_{\e_{n}, y_{n}}|^{q}\right) \, dx\nonumber \\
&=\int_{\R^{N}} f(t_{\e_{n}}\Upsilon_{\e_{n}}) t_{\e_{n}}\Upsilon_{\e_{n}} dx \nonumber\\
&=\int_{\R^{N}} f(t_{\e_{n}} \psi(|\e_{n}z|) \omega(z)) t_{\e_{n}} \psi(|\e_{n}z|) \omega(z) \, dz.
\end{align}
Now, let us prove that $t_{\e_{n}}\rightarrow 1$. Firstly we show that $t_{\e_{n}}\rightarrow t_{0}<\infty$. Assume by contradiction that $|t_{\e_{n}}|\rightarrow \infty$. Then, using the fact that $\psi=1$ in $\B_{\frac{\delta}{2}}(0)\subset \B_{\frac{\delta}{2\e_{n}}}(0)$ for $n$ sufficiently large, we can see that \eqref{4.411} and $(f_5)$ give
\begin{align}\label{4.44}
&t_{\e_{n}}^{p-q}[\Upsilon_{\e_{n}, y_{n}}]_{s, p}^{p}+ [\Upsilon_{\e_{n}, y_{n}}]_{s, q}^{q} + \int_{\R^{N}} V(\e_{n}x) \left(  t_{\e_{n}}^{p-q}|\Upsilon_{\e_{n}, y_{n}}|^{p} +|\Upsilon_{\e_{n}, y_{n}}|^{q}\right) \, dx\nonumber \\
&\quad \geq \int_{\B_{\frac{\delta}{2}}(0)} \frac{f(t_{\e_{n}} \omega(z))}{(t_{\e_{n}} \omega(z))^{q-1}} |\omega(z)|^{q} dz \geq \frac{f(t_{\e_{n}} \omega(\bar{z}))}{(t_{\e_{n}} \omega(\bar{z}))^{q-1}} \int_{\B_{\frac{\delta}{2}}(0)} |\omega(z)|^{q}dz
\end{align}
where $\bar{z}$ is such that $\omega(\bar{z})=\min\{\omega(z): |z|\leq \frac{\delta}{2}\}>0$.
Putting together $(f_4)$, $p<q$, $t_{\e_{n}}\rightarrow \infty$, \eqref{4.421} and \eqref{4.422} we can see that \eqref{4.44} implies that $\|\Upsilon_{\e_{n}, y_{n}}\|_{V, q}^{q}\rightarrow \infty$ which gives a contradiction.
Therefore, up to a subsequence, we may assume that $t_{\e_{n}}\rightarrow t_{0}\geq 0$.
If $t_{0}=0$, we can use \eqref{4.421}, \eqref{4.422}, \eqref{4.411}, $p<q$ and $(f_2)$, to get
$$
\|\Upsilon_{\e_{n}, y_{n}}\|_{V, p}^{p}\rightarrow 0,
$$
that is a contradiction. Hence, $t_{0}>0$.
Now, we show that $t_{0}=1$.
Taking the limit as $n\rightarrow \infty$ in \eqref{4.411}, we can see that
\begin{align}\label{TVV1}
t_{0}^{p-q}[\omega]^{p}_{s, p}+ [\omega]^{q}_{s, q}+\int_{\R^{N}} V_{0} (t_{0}^{p-q}|\omega|^{p} dx + |\omega|^{q}) \, dx =\int_{\R^{N}} \frac{f(t_{0}\omega)}{(t_{0}\omega)^{q-1}} \omega^{q} \, dx.
\end{align}
Since $\omega\in  \N_{V_{0}}$ we have
\begin{align}\label{TVV2}
[\omega]^{p}_{s, p}+ [\omega]^{q}_{s, q}+\int_{\R^{N}} V_{0} (|\omega|^{p} dx + |\omega|^{q}) \, dx =\int_{\R^{N}} f(\omega)\omega\, dx.
\end{align}
Putting together \eqref{TVV1} and \eqref{TVV2} we find
\begin{align}\label{TVV1}
(t_{0}^{p-q}-1)[\omega]^{p}_{s, p}+(t_{0}^{p-q}-1) \int_{\R^{N}} V_{0} |\omega|^{p}\, dx =\int_{\R^{N}} \left(\frac{f(t_{0}\omega)}{(t_{0}\omega)^{q-1}}- \frac{f(\omega)}{\omega^{q-1}}\right) \omega^{q}\, dx.
\end{align}
By $(f_{5})$, we can deduce that $t_{0}=1$.
This fact and the Dominated Convergence Theorem yield
\begin{equation}\label{4.45}
\lim_{n\rightarrow \infty}\int_{\R^{N}} F(t_{\e_{n}} \Upsilon_{\e_{n}, y_{n}})\, dx=\int_{\R^{N}} F(\omega)\, dx.
\end{equation}
Hence, taking the limit as $n\rightarrow \infty$ in
\begin{align*}
\J_{\e}(\Phi_{\e_{n}}(y_{n}))=& \frac{t_{\e_{n}}^{p}}{p}[\Upsilon_{\e_{n}, y_{n}}]_{s, p}^{p}+ \frac{t_{\e_{n}}^{q}}{q}[\Upsilon_{\e_{n}, y_{n}}]_{s, q}^{q} \\
&+ \int_{\R^{N}} V(\e_{n}x) \left( \frac{t_{\e_{n}}^{p}}{p} |\Upsilon_{\e_{n}, y_{n}}|^{p} + \frac{t_{\e_{n}}^{q}}{q} |\Upsilon_{\e_{n}, y_{n}}|^{q}\right) \, dx\\
&  -\int_{\R^{N}} F(t_{\e_{n}} \Upsilon_{\e_{n}, y_{n}}) \, dx
\end{align*}
and exploiting \eqref{4.421}, \eqref{4.422} and \eqref{4.45}, we can deduce that
$$
\lim_{n\rightarrow \infty} \J_{\e_{n}}(\Phi_{\e_{n}}(y_{n}))=\J_{V_{0}}(\omega)=c_{V_{0}}
$$
which is impossible in view of \eqref{4.41}.
\end{proof}

For any $\delta>0$, let $\rho=\rho(\delta)>0$ be such that $M_{\delta}\subset \B_{\rho}(0)$, and let $\chi:\R^{N}\ri \R^{N}$ be defined as
\begin{equation*}
\chi(x)=
\left\{
\begin{array}{ll}
x &\mbox{ if } |x|<\rho \\
\dfrac{\rho x}{|x|} &\mbox{ if } |x|\geq \rho.
\end{array}
\right.
\end{equation*}
Finally, let us consider the map $\beta_{\e}: \N_{\e}\ri \R^{N}$ given by
\begin{align*}
\beta_{\e}(u)=\frac{ {\int_{\R^{N}} \chi(\e x) \left(|u|^{p}+ |u|^{q}\right) \, dx}}{ {\int_{\R^{N}} \left(|u|^{p}+ |u|^{q}\right) \, dx}}.
\end{align*}

\begin{lemma}\label{lem10}
The functional $\Phi_{\e}$ verifies the following limit
\begin{align*}
\lim_{\e\ri 0} \beta_{\e}(\Phi_{\e}(y))=y \mbox{ uniformly in } y\in M.
\end{align*}
\end{lemma}

\begin{proof}
Suppose by contradiction that there exist $\delta_{0}>0$, $\{y_{n}\}\subset M$ and $\e_{n}\rightarrow 0$ such that
\begin{equation}\label{4.4}
|\beta_{\e_{n}}(\Phi_{\e_{n}}(y_{n}))-y_{n}|\geq \delta_{0}.
\end{equation}
Using the definitions of $\Phi_{\e_{n}}(y_{n})$, $\beta_{\e_{n}}$, $\psi$ and the change of variable $z= \frac{\e_{n} x-y_{n}}{\e_{n}}$, we can see that
$$
\beta_{\e_{n}}(\Phi_{\e_{n}}(y_{n}))=y_{n}+\frac{\int_{\R^{N}}[\chi(\e_{n}z+y_{n})-y_{n}] (|\psi (|\e_{n}z|) \omega(z)|^{p}+ |\psi(|\e_{n}z|) \omega(z)|^{q})\, dz}{\int_{\R^{N}} (|\psi(|\e_{n}z|) \omega(z)|^{p}+ |\psi(|\e_{n}z|) \omega(z)|^{q})\, dz}.
$$
Taking into account $\{y_{n}\}\subset M\subset \B_{\rho}(0)$ and the Dominated Convergence Theorem, we can infer that
$$
|\beta_{\e_{n}}(\Phi_{\e_{n}}(y_{n}))-y_{n}|=o_{n}(1)
$$
which contradicts (\ref{4.4}).
\end{proof}

Let $h:[0, \infty) \ri [0, \infty)$ be such that $h(\e)\ri 0$ as $\e\ri 0$. Let $\widetilde{\N}_{\e}=\{u\in \N_{\e} : \J_{\e}(u)\leq c_{V_{0}} + h(\e)\}$.
\begin{lemma}\label{lem12}
Let $\delta>0$ and $M_{\delta}= \{x\in \R^{N} : dist (x, M)\leq \delta\}$. Then
\begin{align*}
\lim_{\e \ri 0} \sup_{u\in \tilde{\N}_{\e}} \inf_{y\in M_{\delta}} |\beta_{\e}(u)-y|=0.
\end{align*}
\end{lemma}

\begin{proof}
Let $\e_{n}\rightarrow 0$ as $n\rightarrow \infty$. For any $n\in \mathbb{N}$, there exists $\{u_{n}\}\subset \widetilde{\N}_{\e_{n}}$ such that
$$
\sup_{u\in \widetilde{\N}_{\e_{n}}} \inf_{y\in M_{\delta}}|\beta_{\e_{n}}(u)-y|=\inf_{y\in M_{\delta}}|\beta_{\e_{n}}(u_{n})-y|+o_{n}(1).
$$
Therefore, it suffices to prove that there exists $\{y_{n}\}\subset M_{\delta}$ such that
\begin{equation}\label{3.13}
\lim_{n\rightarrow \infty} |\beta_{\e_{n}}(u_{n})-y_{n}|=0.
\end{equation}
Thus, recalling that $\{u_{n}\}\subset  \widetilde{\N}_{\e_{n}}\subset  \N_{\e_{n}}$, we deduce that
$$
c_{V_{0}}\leq c_{\e_{n}}\leq  \I_{\e_{n}}(u_{n})\leq c_{V_{0}}+h(\e_{n})
$$
which implies that $\I_{\e_{n}}(u_{n})\rightarrow c_{V_{0}}$. By Proposition \ref{prop6}, there exists $\{\tilde{y}_{n}\}\subset \R^{N}$ such that $y_{n}=\e_{n}\tilde{y}_{n}\in M_{\delta}$ for $n$ sufficiently large. Thus
$$
\beta_{\e_{n}}(u_{n})=y_{n}+\frac{ {\int_{\R^{N}}[\chi(\e_{n}z+y_{n})-y_{n}] (|u_{n}(z+\tilde{y}_{n})|^{p}+ |u_{n}(z+\tilde{y}_{n})|^{q}) \, dz}}{ {\int_{\R^{N}} (|u_{n}(z+\tilde{y}_{n})|^{p}+|u_{n}(z+\tilde{y}_{n})|^{q}) \, dz}}.
$$
Since $u_{n}(\cdot+\tilde{y}_{n})$ strongly converges in $\X_{V_{0}}$ and $\e_{n}z+y_{n}\rightarrow y\in M$, we deduce that $\beta_{\e_{n}}(u_{n})=y_{n}+o_{n}(1)$, that is (\ref{3.13}) holds true.
\end{proof}

\begin{proof}[Proof of Theorem \ref{thm1}]
By Lemmas \ref{lem9} and \ref{lem12}, we have that $\beta_{\e}\circ \Phi_{\e}$ is homotopic to the inclusion map $id: M\ri M_{\delta}$, which implies that
\begin{align*}
cat_{\tilde{\N}_{\e}} (\tilde{\N}_{\e}) \geq cat_{M_{\delta}}(M).
\end{align*}
Since the functional $\J_{\e}$ satisfies the Palais-Smale condition at level $c\in (c_{0}, c_{0}+h(\e))$, by Ljusternick-Schnirelmann theory of critical points we can conclude that $\J_{\e}$ has at least $cat_{M_{\delta}}(M)$ critical points on $\N_{\e}$. Therefore, by Corollary \ref{cor12}, $\J_{\e}$ has at least $cat_{M_{\delta}}(M)$ critical points in $\X_{\e}$.
\end{proof}

\section{Regularity of solutions to \eqref{P}} \label{Sect6}

\noindent
This last section deals with the regularity of nonnegative solutions to \eqref{P}. More precisely, using a Moser iteration argument \cite{Moser} we are able to prove the following result.

\begin{lemma}\label{lemMoser}
Let $u\in \X_{\e}$ be a nonnegative weak solution to \eqref{P}. Then $u\in L^{\infty}(\R^{N})$ and there exists $K>0$ such that $|u|_{\infty}\leq K$.
\end{lemma}

\begin{proof}
For any $L>0$ and $\beta>1$, we take
\begin{equation*}
\varphi(u)= uu_{L}^{q(\beta-1)} \in \X_{\e},
\end{equation*}
where $u_{L}= \min \{u, L\}$, as test function in \eqref{P} and we have
\begin{align*}
&\iint_{\R^{2N}} \frac{|u(x)- u(y)|^{p-2} (u(x)- u(y)) ((uu_{L}^{q(\beta-1)}) (x) - (u u_{L}^{q(\beta-1)})(y))}{|x- y|^{N+sp}} \, dxdy\\
&+ \iint_{\R^{2N}} \frac{|u(x)- u(y)|^{q-2} (u(x)- u(y)) ((uu_{L}^{q(\beta-1)}) (x) - (uu_{L}^{q(\beta-1)})(y))}{|x- y|^{N+sp}} \, dxdy\\
&+ \int_{\R^{N}} V(\e x) |u|^{p} u_{L}^{q(\beta-1)}\, dx+ \int_{\R^{N}} V(\e x) |u|^{q} u_{L}^{q(\beta-1)}\, dx  \\
&= \int_{\R^{N}} f(u) u u_{L}^{q(\beta-1)} \, dx.
\end{align*}
Taking into account \eqref{V0} and \eqref{growthf}, choosing $\xi \in (0, V_{0})$, we obtain
\begin{align}\label{AL2}
&\iint_{\R^{2N}} \frac{|u(x)- u(y)|^{p-2} (u(x)- u(y)) ((uu_{L}^{q(\beta-1)}) (x) - (u u_{L}^{q(\beta-1)})(y))}{|x- y|^{N+sp}} \, dxdy\nonumber \\
&+ \iint_{\R^{2N}} \frac{|u(x)- u(y)|^{q-2} (u(x)- u(y)) ((uu_{L}^{q(\beta-1)}) (x) - (uu_{L}^{q(\beta-1)})(y))}{|x- y|^{N+sp}} \, dxdy\nonumber \\
&\leq C \int_{\R^{N}} |u|^{\q} u_{L}^{q(\beta-1)} \, dx.
\end{align}
Let us define
\begin{align*}
\Lambda(t)= \frac{|t|^{q}}{q} \quad \mbox{ and } \quad \Gamma(t)= \int_{0}^{t} (\varphi'(\tau))^{\frac{1}{q}} d\tau.
\end{align*}
Since $\varphi$ is an increasing function, we can infer
$$
(a-b)(\varphi(a)- \varphi(b))\geq 0 \quad \mbox{ for any } a, b\in \R.
$$
Using this last inequality and applying Jensen's inequality, we obtain
\begin{equation*}\label{AL1}
\Lambda'(a-b)(\varphi(a)- \varphi(b)) \geq |\Gamma(a)- \Gamma(b)|^{q} \quad \mbox{ for any } a, b\in \R,
\end{equation*}
from which it follows that
\begin{equation*}
|\Gamma(u)(x) \!-\! \Gamma(u)(y)|^{q} \leq |u(x)\!-\! u(y)|^{q-2} (u(x)\!-\! u(y)) ((uu_{L}^{q(\beta-1)}) (x) - (u u_{L}^{q(\beta-1)})(y)).
\end{equation*}
We can also note that $\Gamma (u)\geq \frac{1}{\beta} u u_{L}^{\beta-1}$. Thus, by Sobolev inequality we deduce that
\begin{align}\label{AL3}
&\iint_{\R^{2N}} \frac{|u(x)- u(y)|^{q-2} (u(x)- u(y)) ((uu_{L}^{q(\beta-1)}) (x) - (u u_{L}^{q(\beta-1)})(y))}{|x- y|^{N+sp}} \, dxdy \nonumber \\
&\quad \geq [\Gamma (u)]_{s, q}^{q} \geq S_{*} |\Gamma (u)|_{\q}^{q} \geq \frac{S_{*}}{\beta^{q}} |u u_{L}^{\beta-1}|_{\q}^{q}.
\end{align}
Moreover,
\begin{align}\label{AL4}
&\iint_{\R^{2N}} \frac{|u(x)- u(y)|^{p-2} (u(x)- u(y)) ((uu_{L}^{q(\beta-1)}) (x) - (u u_{L}^{q(\beta-1)})(y))}{|x- y|^{N+sp}} \, dxdy\nonumber \\
&=\iint_{\R^{2N}} |u(x)- u(y)|^{p-2} (u(x)- u(y)) \nonumber \\
&\qquad \cdot \frac{ [ (u(x) - u(y)) u_{L, n}^{q(\beta-1)}(x) + u(y) (u_{L}^{q(\beta-1)}(x)- u_{L}^{q(\beta-1)}(y))]}{|x- y|^{N+sp}}\, dxdy\\
&= \iint_{\R^{2N}} \frac{|u(x)- u(y)|^{p}}{|x- y|^{N+sp}}\, u_{L}^{q(\beta-1)}(x)\, dxdy\nonumber \\
&+ \iint_{\R^{2N}} \frac{|u(x)- u(y)|^{p-2} \, (u(x)- u(y)) \,u(y) \,(u_{L}^{q(\beta-1)}(x)- u_{L}^{q(\beta-1)}(y))}{|x- y|^{N+sp}}\, dxdy\geq 0.
\end{align}
Indeed
{\small\begin{align*}
&\iint_{\R^{2N}} \frac{|u(x)- u(y)|^{p-2} \, (u(x)- u(y)) \,u(y) \,(u_{L}^{q(\beta-1)}(x)- u_{L}^{q(\beta-1)}(y))}{|x- y|^{N+sp}}\, dxdy\\
&\displaystyle= \int_{\{u(x) \geq L\}} \int_{\{u(y)\leq L\}} \frac{|u(x)- u(y)|^{p-2} \, (u(x)- u(y)) \,u(y) \,(u_{L}^{q(\beta-1)}(x)- u_{L}^{q(\beta-1)}(y))}{|x- y|^{N+sp}}\, dxdy\\
&+ \int_{\{u(x) \leq L\}} \int_{\{u(y)\leq L\}} \frac{|u(x)- u(y)|^{p-2} \, (u(x)- u(y)) \,u(y) \,(u_{L}^{q(\beta-1)}(x)- u_{L}^{q(\beta-1)}(y))}{|x- y|^{N+sp}}\, dxdy\\
&+ \int_{\{u(x) \geq L\}} \int_{\{u(y)\geq L\}} \frac{|u(x)- u(y)|^{p-2} \, (u(x)- u(y)) \,u(y) \,(u_{L}^{q(\beta-1)}(x)- u_{L}^{q(\beta-1)}(y))}{|x- y|^{N+sp}}\, dxdy\\
&+ \int_{\{u(x) \leq L\}} \int_{\{u(y)\geq L\}} \frac{|u(x)- u(y)|^{p-2} \, (u(x)- u(y)) \,u(y) \,(u_{L}^{q(\beta-1)}(x)- u_{L}^{q(\beta-1)}(y))}{|x- y|^{N+sp}}\, dxdy\\
&=:I+II+III+IV.
\end{align*}}
Note that $III=0$. Moreover $I\geq 0$ because when $u(x) \geq L$ and $u(y)\leq L$ we have
\begin{align*}
u(x)- u(y)\geq u(x)- L\geq 0
\end{align*}
 and
 \begin{align*}  u_{L}^{q(\beta-1)}(x)- u_{L}^{q(\beta-1)}(y)= L^{q(\beta-1)} -u^{q(\beta-1)}(y)\geq 0.
\end{align*}
On the other hand, when $u(x)\leq L$ and $u(y)\leq L$ we can see that
\begin{align*}
&(u(x)- u(y)) (u_{L}^{q(\beta-1)}(x)- u_{L}^{q(\beta-1)}(y))\\
=& (u(x)- u(y)) (u^{q(\beta-1)}(x)- u^{q(\beta-1)}(y)) \geq 0,
\end{align*}
then $II\geq 0$. Finally, when $u(x)\leq L$ and $u(y)\geq L$ we can infer that
\begin{align*}
u(x)- u(y)\leq L- u(y)\leq 0
\end{align*}
 and
 \begin{align*}  u_{L}^{q(\beta-1)}(x)- u_{L}^{q(\beta-1)}(y)= u^{q(\beta-1)}(x)- L^{q(\beta-1)}\leq 0,
\end{align*}
thus $IV\geq 0$. Combining \eqref{AL2}, \eqref{AL3} and \eqref{AL4} we obtain
\begin{align}\label{AL5}
|uu_{L}^{\beta-1}|_{\q}^{q} \leq C \beta^{q} \int_{\R^{N}} u^{\q} u_{L}^{q(\beta-1)} dx.
\end{align}
Take $\beta= \frac{\q}{q}$ and fix $R>0$. Then,
\begin{align*}
\int_{\R^{N}} u^{\q} u_{L}^{\q-q} \, dx& = \int_{\R^{N}}  u^{\q-q} \left(u u_{L}^{\frac{\q-q}{q}}\right)^{q} \, dx\\
&= \int_{\{u\leq R\}} u^{\q-q} \left(u u_{L}^{\frac{\q-q}{q}}\right)^{q} \, dx+ \int_{\{u\geq R\}} u^{\q-q} \left(u u_{L}^{\frac{\q-q}{q}}\right)^{q} \, dx\\
&=:I_{1}+I_{2}
\end{align*}
Recalling that $0\leq u_{L}\leq u$, we can see that
\begin{align*}
I_{1} \leq \int_{\{u\leq R\}} R^{\q-q} u^{\q} \, dx.
\end{align*}
Applying H\"older inequality with $\frac{\q}{\q-q}$ and $\frac{\q}{q}$ we have
\begin{align*}
I_{2}&\leq \left( \int_{\{u\geq R\}} u^{\q}\, dx \right)^{\frac{\q-q}{\q}} \left( \int_{\R^{N}} (u u_{L}^{\frac{\q-q}{\q}})^{\q} \, dx \right)^{\frac{q}{\q}}.
\end{align*}
Let us note that $u\in L^{\q}(\R^{N})$, then for $R$ sufficiently large we can infer
\begin{align*}
\left( \int_{\{u\geq R\}} u^{\q}\, dx \right)^{\frac{\q-q}{\q}} \leq \eta \beta^{-q}.
\end{align*}
Thus we have
\begin{align*}
I_{2} \leq \eta \beta^{-q}  \left( \int_{\R^{N}} (u u_{L}^{\frac{\q-q}{\q}})^{\q} \, dx \right)^{\frac{q}{\q}}.
\end{align*}
Putting together the estimates for $I_{1}$ and $I_{2}$ we obtain
\begin{align}\label{AL6}
\int_{\R^{N}} u^{\q} u_{L}^{\q-q} \, dx\leq \int_{\R^{N}} R^{\q-q} u^{\q} \, dx+ \eta \beta^{-q}  \left( \int_{\R^{N}} (u u_{L}^{\frac{\q-q}{\q}})^{\q} \, dx \right)^{\frac{q}{\q}}.
\end{align}
Combining \eqref{AL5} and \eqref{AL6} we can infer that
\begin{align*}
\left( \int_{\R^{N}} (u u_{L}^{\frac{\q-q}{q}})^{\q} \, dx \right)^{\frac{q}{\q}} &\leq C \beta^{q} \int_{\R^{N}} R^{\q-q} u^{\q} \, dx+ C \eta \left( \int_{\R^{N}} (u u_{L}^{\frac{\q-q}{\q}})^{\q} \, dx \right)^{\frac{q}{\q}}.
\end{align*}
Choosing $\eta <\frac{1}{C}$ we have
\begin{align*}
\left( \int_{\R^{N}} (u u_{L}^{\frac{\q-q}{q}})^{\q} \, dx \right)^{\frac{q}{\q}} &\leq \bar{C} \beta^{q} \int_{\R^{N}} R^{\q-q} u^{\q} \, dx<\infty,
\end{align*}
and taking the limit as $L\ri \infty$ we deduce that $u \in L^{\frac{(\q)^{2}}{q}}(\R^{N})$.

Since $0\leq u_{L}\leq u$ and passing to the limit as $L\ri \infty$ in \eqref{AL5} we have
\begin{align*}
|u|_{\beta \q}^{\beta q} \leq C \beta^{q} \int_{\R^{N}} u^{\q+q(\beta-1)} \, dx,
\end{align*}
from which we deduce that
\begin{align}\label{AL7}
\left( \int_{\R^{N}} u^{\beta \q}\, dx \right)^{\frac{1}{q(\beta-1)}} \leq (\bar{C} \beta)^{\frac{1}{\beta-1}} \left( \int_{\R^{N}} u^{\q+ q(\beta-1)}\, dx \right)^{\frac{1}{q(\beta-1)}}.
\end{align}

For $m\geq 1$ we set
\begin{equation*}
\q+ q(\beta_{m+1}-1)= \beta_{m}\,\q \quad \mbox{ and } \quad \beta_{1}= \frac{\q}{q}.
\end{equation*}
In particular
\begin{equation*}
\beta_{m+1}= \beta_{1}^{m}(\beta_{1}-1)+1,
\end{equation*}
and $\lim_{m\ri \infty} \beta_{m}= \infty$. Let us define $\displaystyle{\Psi_{m}:= \left( \int_{\R^{N}} u^{\q \beta_{m}} \, dx \right)^{\frac{1}{\q(\beta_{m}-1)}}}$. Then \eqref{AL7} becomes
\begin{align*}
\Psi_{m+1}\leq (\bar{C} \beta_{m+1})^{\frac{1}{\beta_{m+1} -1}} \Psi_{m}.
\end{align*}
Hence, we can find $C_{0}>0$ independent of $m$ such that
\begin{align*}
\Psi_{m+1}\leq \prod_{k=1}^{m} (\bar{C} \beta_{k+1})^{\frac{1}{\beta_{k+1} -1}} \, \Psi_{1}\leq C_{0} \Psi_{1}.
\end{align*}
Taking the limit as $m\ri \infty$ we get $|u|_{\infty}\leq K$.
\end{proof}

\section*{Acknowledgements}
C. O. Alves was partially supported by CNPq/Brazil  Proc. 304804/2017-7


\begin{thebibliography}{77}


\bibitem{Alv}
C.O. Alves, 
{\it Existence of positive solutions for a problem with lack of compactness involving the $p$-Laplacian},  
Nonlinear Anal. {\bf 51} (2002), no. 7, 1187--1206.

\bibitem{AlvAmb}
C.O. Alves and V. Ambrosio, 
{\it A multiplicity result for a nonlinear fractional Schr\"odinger equation in $\R^{N}$ without the Ambrosetti-Rabinowitz condition}, J. Math. Anal. Appl. {\bf 466} (2018), no. 1, 498--522.


\bibitem{AFpq}
C.O. Alves and G.M. Figueiredo, 
{\it Multiplicity and concentration of positive solutions for a class of quasilinear problems},
Adv. Nonlinear Stud. {\bf 11} (2011), no. 2, 265--294. 

\bibitem{AM}
C.O. Alves and O.H. Miyagaki, 
{\it Existence and concentration of solution for a class of fractional elliptic equation in $\R^{N}$ via penalization method}, Calc. Var. Partial Differential Equations {\bf 55} (2016), no. 3, Art. 47, 19 pp.

\bibitem{AP}
C.O. Alves and M.T.O. Pimenta, 
{\it On existence and concentration of solutions to a class of quasilinear problems involving the 1-Laplace operator}, 
Calc. Var. Partial Differential Equations {\bf 56} (2017), no. 5, Art. 143, 24 pp.

\bibitem{AT}
C.O. Alves and C.L. Torres, 
{\it Existence and concentration of solution for a non-local regional Schr\"odinger equation with competing potentials},
to appear in Glasgow Mathematical Journal.

\bibitem{AR}
A. Ambrosetti and P.H. Rabinowitz,
{\it Dual variational methods in critical point theory and applications}, 
J. Funct. Anal. {\bf 14} (1973), 349--381.

\bibitem{A1}
V. Ambrosio, 
{\it Multiple solutions for a fractional $p$-Laplacian equation with sign-changing potential}, 
Electron. J. Diff. Equ., vol. 2016 (2016), no. 151, pp. 1--12.

\bibitem{A3}
V. Ambrosio,
{\it Multiplicity of positive solutions for a class of fractional Schr\"odinger equations via penalization method}, 
Ann. Mat. Pura Appl. (4) {\bf 196} (2017), no. 6, 2043--2062.

\bibitem{A4}
V. Ambrosio, 
{\it Concentrating solutions for a class of nonlinear fractional Schr\"odinger equations in $\R^{N}$}, 
Rev. Mat. Iberoam. (in press), arXiv:1612.02388.

\bibitem{Apq}
V. Ambrosio, 
{\it Fractional $p\&q$ Laplacian problems in $\R^{N}$ with critical growth}, 
Preprint. 	arXiv:1801.10449.

\bibitem{A2} 
V. Ambrosio,
{\it A multiplicity result for a fractional $p$-Laplacian problem without growth conditions},
Riv. Math. Univ. Parma (N.S.) \textbf{9} (2018), 53--71.

\bibitem{AH}
V. Ambrosio and H. Hajaiej, 
{\it Multiple solutions for a class of nonhomogeneous fractional Schr\"odinger equations in $\R^{N}$}, 
J. Dynam. Differential Equations {\bf 30} (2018), no. 3, 1119--1143. 

\bibitem{AI}
V. Ambrosio and T. Isernia, 
{\it Concentration phenomena for a fractional Schr\"odinger-Kirchhoff type equation},  Math. Methods Appl. Sci. {\bf 41} (2018), no. 2, 615--645. 

\bibitem{AI2}
V. Ambrosio and T. Isernia, 
{\it Sign-changing solutions for a class of Schr\"odinger equations with vanishing potentials},  
Rend. Lincei Mat. Appl.  {\bf 29} (2018), 127--152. 

\bibitem{AI3}
V. Ambrosio and T. Isernia, 
{\it Multiplicity and concentration results for some nonlinear Schr\"odinger equations with the fractional $p$-Laplacian}, Discrete Contin. Dyn. Syst. {\bf 38} (2018), no.11, 5835--5881.

\bibitem{BF}
S. Barile and G. M. Figueiredo,
{\it Existence of a least energy nodal solution for a class of $p\&q$-quasilinear elliptic equations}, 
Adv. Nonlinear Stud. {\bf 14} (2014), no. 2, 511--530.

\bibitem{BL}
H. Br\'ezis and E.H. Lieb, 
{\it A relation between pointwise convergence of functions and convergence of functionals}, 
Proc. Amer. Math. Soc. {\bf 88} (1983), no. 3, 486--490.

\bibitem{CB}
C. Chen and J. Bao, 
{\it Existence, nonexistence, and multiplicity of solutions for the fractional $p\&q$-Laplacian equation in $\R^{N}$}, 
Bound. Value Probl. 2016, Paper No. 153, 16 pp.


\bibitem{CIL}
L. Cherfils and V. Il'yasov, 
{\it On the stationary solutions of generalized reaction diffusion equations with $p\&q$-Laplacian},
Commun. Pure Appl. Anal. {\bf 1} (4), 1--14 (2004).

\bibitem{DKP1}
A. Di Castro, T. Kuusi and G. Palatucci, 
{\it Nonlocal Harnack inequalities},
J. Funct. Anal. {\bf 267} (2014), no. 6, 1807--1836.


\bibitem{DKP2}
A. Di Castro, T. Kuusi and G. Palatucci,
{\it Local behavior of fractional $p$-minimizers}, 
Ann. Inst. H. Poincar\'e Anal. Non Lin\'eaire {\bf 33} (2016), no. 5, 1279--1299. 

\bibitem{DPV}
E. Di Nezza, G. Palatucci and E. Valdinoci, 
{\it Hitchhiker's guide to the fractional Sobolev spaces}, 
Bull. Sci. math. {\bf 136} (2012), 521--573.

\bibitem{DPMV}
S. Dipierro, M. Medina and E. Valdinoci, 
{\it Fractional elliptic problems with critical growth in the whole of $\R^{n}$}, 
Appunti. Scuola Normale Superiore di Pisa (Nuova Serie) [Lecture Notes. Scuola Normale Superiore di Pisa (New Series)], 15. Edizioni della Normale, Pisa, 2017. viii+152 pp.

\bibitem{Ekeland}
I. Ekeland, 
{\it On the variational principle},
J. Math. Anal. Appl. {\bf 47} (1974), 324--353. 

\bibitem{FQT} 
P. Felmer, A Quass and J. Tan, {\it Positive solutions of nonlinear Schr\"odinger equation with the fractional Laplacian}, Proc. Royal  Soc. Edinburgh A {\bf 142} (2012), 1237--1262.

\bibitem{Fig1}
G.M. Figueiredo, 
{\it Existence of positive solutions for a class of $p\&q$ elliptic problems with critical growth on $\R^N$}, 
J. Math. Anal. Appl. {\bf 378} (2011) 507--518.

\bibitem{Fig2}
G.M. Figueiredo, 
{\it Existence and multiplicity of solutions for a class of $p\&q$ elliptic problems with critical exponent}, 
Math. Nachr. {\bf 286} (11--12) (2013) 1129--1141.

\bibitem{FS} 
G.M. Figueiredo and G. Siciliano, {\it A multiplicity result via Ljusternick-Schnirelmann category and Morse theory for a fractional Schr\"odinger equation in $\R^{N}$},  
NoDEA Nonlinear Differential Equations Appl. {\bf 23} (2016), no. 2, Art. 12, 22 pp.


\bibitem{FP}
A. Fiscella and P. Pucci,
{\it Kirchhoff-Hardy fractional problems with lack of compactness}, 
Adv. Nonlinear Stud. {\bf 17} (2017), no. 3, 429--456.


\bibitem{FrP}
G. Franzina and G. Palatucci,
{\it Fractional p-eigenvalues},
Riv. Math. Univ. Parma (N.S.) {\bf 5} (2014), no. 2, 373--386.

\bibitem{HL}
C. He and G. Li,
{\it The regularity of weak solutions to nonlinear scalar field elliptic equations containing $p\&q$-Laplacians}, 
Ann. Acad. Sci. Fenn. Math. {\bf 33} (2008), no. 2, 337--371.


\bibitem{IMS}
A. Iannizzotto, S. Mosconi and M. Squassina,
{\it Global H\"older regularity for the fractional $p$-Laplacian},  
Rev. Mat. Iberoam. {\bf 32} (2016), no. 4, 1353--1392. 

\bibitem{J}
L. Jeanjean, 
{\it On the existence of bounded Palais-Smale sequences and application to a Landesman- Lazer type problem set on $\R^{N}$},
Proc. Roy. Soc. Edinburgh Sect.A, \textbf{129} (1999), 787--809.


\bibitem{KMS}
T. Kuusi, G. Mingione and Y. Sire, 
{\it Nonlocal Equations with Measure Data}, 
Comm. Math. Phys., {\bf 337}(2015) (3) 1317--1368.

\bibitem{Laskin1} 
N. Laskin, 
{\it Fractional quantum mechanics and L\`evy path integrals}, 
Phys. Lett. A {\bf 268} (2000), 298--305.

\bibitem{LG}
G. Li and Z. Guo, 
{\it Multiple solutions for the $p\&q$-Laplacian problem with critical exponent}, 
Acta Math. Sci. Ser. B Engl. Ed. {\bf 29} (4) (2009) 903--918.

\bibitem{LL}
G.B. Li and X. Liang,
{\it The existence of nontrivial solutions to nonlinear elliptic equation of $p$-$q$-Laplacian type on $\R^N$},
Nonlinear Anal. {\bf 71} (2009) 2316--2334.

\bibitem{LLq}
E. Lindgren and P. Lindqvist, 
{\it Fractional eigenvalues},
Calc. Var. {\bf 49} (2014) 795--826.


\bibitem{LW}
Z. Liu and Z. Wang, 
{\it On the Ambrosetti-Rabinowitz Superlinear Condition}, 
Advanced Nonlinear Studies, {\bf 4} (2004), (4), 563--574.

\bibitem{MMB}
J. Mawhin and G. Molica Bisci,
{\it A Brezis-Nirenberg type result for a nonlocal fractional operator},  
J. Lond. Math. Soc. (2) {\bf 95} (2017), no. 1, 73--93.

\bibitem{MP}
E. S. Medeiros and K. Perera, 
{\it Multiplicity of solutions for a quasilinear elliptic problem via the cohomoligical index},
Nonlinear Anal. {\bf 71} (2009), 3654--3660.


\bibitem{MeW}
C. Mercuri and M. Willem, 
{\it A global compactness result for the $p$-Laplacian involving critical nonlinearities},
Discrete Contin. Dyn. Syst. {\bf 28} (2010), no. 2, 469--493. 

\bibitem{MS}
O. Miyagaki and M. Souto, 
{\it Superlinear problems without Ambrosetti and Rabinowitz growth condition},
J. Differential Equations, \textbf{245} (2008), 3628--3638.

\bibitem{MBRS}
G. Molica Bisci, V. R\u{a}dulescu and R. Servadei,
{\it Variational Methods for Nonlocal Fractional Problems},
{\em Cambridge University Press}, \textbf{162} Cambridge, 2016.

\bibitem{Moser}
J. Moser,
{\it A new proof of De Giorgi's theorem concerning the regularity problem for elliptic differential equations}, 
Comm. Pure Appl. Math. {\bf 13} (1960), 457--468.


\bibitem{PP}
G. Palatucci and A. Pisante,
{\it Improved Sobolev embeddings, profile decomposition, and concentration-compactness for fractional Sobolev spaces}, 
Calc. Var. Partial Differential Equations {\bf 50} (2014), 799--829.

\bibitem{Rab} 
P.H. Rabinowitz,  
{\it On a class of nonlinear Schr\"{o}dinger equations}, Z. Angew Math. Phys. {\bf{43}} (1992), 270--291.

\bibitem{SZ}
M. Schechter and W. Zou,
{\it Superlinear problems},
Pacific J. Math., \textbf{214} (2004), 145--160.

\bibitem{Secchi} 
S. Secchi, 
\textit{Ground state solutions for nonlinear fractional Schr\"odinger equations in $\R^N$,}  
J. Math. Phys.   {\bf 54}  (2013), 031501-17 pages.

\bibitem{SW}
A. Szulkin and T. Weth, 
{\it The method of Nehari manifold}, 
in Handbook of non-convex analysis and applications, 597--632, Int. Press, Somerville, MA, (2010).

\bibitem{Torres}
C.E. Torres Ledesma,  
{\it Existence and symmetry result for fractional p-Laplacian in $\R^{n}$}, 
Commun. Pure Appl. Anal. {\bf 16} (2017), no. 1, 99--113. 

\bibitem{W}
M. Willem, 
{\it Minimax theorems}, 
Progress in Nonlinear Differential Equations and their Applications, {\bf 24}. Birkh\"auser Boston, Inc., Boston, MA, 1996. x+162 pp. 


\end{thebibliography}
\end{document}